\documentclass{daj}

\usepackage{amsmath}
\usepackage{amsthm}
\usepackage{amssymb}
\usepackage{bbm}
\usepackage{paralist}

\usepackage[nameinlink,capitalise,noabbrev]{cleveref}

\theoremstyle{plain}
\newtheorem{thm}{Theorem}[section]
\crefname{thm}{Theorem}{Theorems}
\theoremstyle{plain}
\newtheorem{lem}[thm]{Lemma}
\crefname{lem}{Lemma}{Lemmas}
\theoremstyle{plain}
\newtheorem{cor}[thm]{Corollary}
\theoremstyle{plain}
\newtheorem*{claim*}{Claim}
\crefname{claim}{Claim}{Claims}
\theoremstyle{definition}
\newtheorem{defn}[thm]{Definition}
\theoremstyle{plain}
\newtheorem{conjecture}[thm]{Conjecture}
\crefname{conjecture}{Conjecture}{Conjectures}
\theoremstyle{plain}

\crefname{prop}{Proposition}{Propositions}
\theoremstyle{definition}

\theoremstyle{definition}

\theoremstyle{plain}
\newtheorem{claim}[thm]{Claim}

\crefformat{equation}{#2(#1)#3}

\crefname{subsection}{Subsection}{Subsections}

\let\originalleft\left
\let\originalright\right
\renewcommand{\left}{\mathopen{}\mathclose\bgroup\originalleft}
\renewcommand{\right}{\aftergroup\egroup\originalright}

\dajAUTHORdetails{%
  title = {An algebraic inverse theorem for the quadratic Littlewood--Offord problem, and an application to Ramsey graphs}, 
  author = {Matthew Kwan and Lisa Sauermann},
  plaintextauthor = {Matthew Kwan, Lisa Sauermann},
  plaintexttitle = {An algebraic inverse theorem for the quadratic Littlewood-Offord problem, and an application to Ramsey graphs}, 
  runningtitle = {An algebraic inverse theorem for the quadratic Littlewood--Offord problem}, 
}   

\dajEDITORdetails{%
   year={2020},
   number={12},
   received={9 October 2019},   
   published={11 August 2020},  
   doi={10.19086/da.14351},       
}   

\begin{document}

\begin{frontmatter}[classification=text]

\title{An algebraic inverse theorem for the quadratic Littlewood--Offord problem, and an application to Ramsey graphs} 

\author[mkwa]{Matthew Kwan\thanks{Research supported in part by SNSF project 178493.}}
\author[lsau]{Lisa Sauermann}

\global\long\def\RR{\mathbb{R}}%
\global\long\def\FF{\mathbb{F}}%
\global\long\def\QQ{\mathbb{Q}}%
\global\long\def\E{\mathbb{E}}%
\global\long\def\Var{\operatorname{Var}}%
\global\long\def\CC{\mathbb{C}}%
\global\long\def\NN{\mathbb{N}}%
\global\long\def\ZZ{\mathbb{Z}}%
\global\long\def\GG{\mathbb{G}}%
\global\long\def\tallphantom{\vphantom{\sum}}%
\global\long\def\tallerphantom{\vphantom{\int}}%
\global\long\def\supp{\operatorname{supp}}%
\global\long\def\one{\mathbbm{1}}%
\global\long\def\d{\operatorname{d}}%
\global\long\def\Unif{\operatorname{Unif}}%
\global\long\def\Po{\operatorname{Po}}%
\global\long\def\Bin{\operatorname{Bin}}%
\global\long\def\Ber{\operatorname{Ber}}%
\global\long\def\Geom{\operatorname{Geom}}%
\global\long\def\Rad{\operatorname{Rad}}%
\global\long\def\floor#1{\left\lfloor #1\right\rfloor }%
\global\long\def\ceil#1{\left\lceil #1\right\rceil }%
\global\long\def\falling#1#2{\left(#1\right)_{#2}}%
\global\long\def\cond{\,\middle|\,}%
\global\long\def\su{\subseteq}%
\global\long\def\row{\operatorname{row}}%
\global\long\def\col{\operatorname{col}}%
\global\long\def\spn{\operatorname{span}}%
\global\long\def\eps{\varepsilon}%
\global\long\def\matrixnorm{1}%
\global\long\def\polynorm{\mathrm C}%

\global\long\def\posreal{s}%
\global\long\def\zint{k}%

\begin{abstract}
\noindent Consider a quadratic polynomial $f\left(\xi_{1},\dots,\xi_{n}\right)$
of independent Bernoulli random variables. What can be said about
the concentration of $f$ on any single value? This generalises the
classical Littlewood--Offord problem, which asks the same question
for linear polynomials. As in the linear case, it is known that the point probabilities of $f$ can be as large as about $1/\sqrt{n}$, but still poorly understood is the ``inverse'' question of characterising the algebraic and arithmetic features $f$ must have if it has point probabilities comparable to this bound. In this paper we prove some results
of an algebraic flavour, showing that if $f$ has point probabilities
much larger than $1/n$ then it must be close to a quadratic form with low rank. We also give an application to Ramsey graphs, asymptotically answering a question of Kwan, Sudakov and Tran.
\end{abstract}
\end{frontmatter}

\section{Introduction}

Consider a random variable of the form $X=a_{1}\xi_{1}+\dots+a_{n}\xi_{n}$,
where $(a_{1},\dots,a_{n})\in\RR^{n}$ is
a sequence of real numbers and $\xi=(\xi_{1},\dots,\xi_{n})\sim\Rad^{n}$
is a sequence of independent Rademacher random variables (meaning that
$\Pr\left(\xi_{i}=1\right)=\Pr\left(\xi_{i}=-1\right)=1/2$ for each
$i$). Broadly speaking, the classical Littlewood--Offord problem
asks for \emph{anti-concentration} estimates for random variables
of this type: what can we say about the maximum probability that $X$
is equal to a single value, or falls in an interval of prescribed
length?

In connection with their work on random polynomials, Littlewood and
Offord~\cite{LO43} first proved that if each $\left|a_{i}\right|\ge1$,
then the small-ball probabilities $\Pr\left(\left|X-x\right|\le1\right)$,
for $x\in\RR$, are bounded by $O\left(\log n/\sqrt{n}\right)$. (here,
and for the rest of the paper, all asymptotics are as $n\to\infty$).
Erd\H os~\cite{Erd45} later proved the optimal upper bound $\binom{n}{\floor{n/2}}/2^{n}=O\left(1/\sqrt{n}\right)$,
and further work by Hal\'asz~\cite{Hal77}, Tao and Vu~\cite{TV10,TV09b},
Rudelson and Vershynin~\cite{RV08} and Nguyen and Vu~\cite{NV11}
explored the relationship between the concentration behaviour of $X$
and the arithmetic structure of the coefficients $(a_{1},\dots,a_{n})$. This work has had
far-reaching consequences: in particular, these Littlewood--Offord-type
theorems were essential tools in some of the landmark results in random
matrix theory (see for example \cite{TV09a,TV09b,Tik18}).

Observe that $a_{1}\xi_{1}+\dots+a_{n}\xi_{n}$ is a linear polynomial
in $\xi=(\xi_1,\dots,\xi_n)$, so a natural variation on the Littlewood--Offord problem
is to consider \emph{quadratic} polynomials in $\xi$. This direction
of research was popularised by Costello, Tao and Vu~\cite{CTV06}
in connection with their proof of Weiss' conjecture that a random
symmetric $\pm1$ matrix typically has full rank, and was further
explored by Costello~\cite{Cos13} and Nguyen~\cite{Ngu12} (higher-degree
polynomials were also considered by Rosi\'nski and Samorodnitsky~\cite{RS96},
Razborov and Viola~\cite{RV13}, Meka, Nguyen and Vu~\cite{MNV16},
and Fox, Kwan and Sauermann~\cite{FKS1}). Specifically, Costello~\cite{Cos13}
proved that if $f$ is a quadratic polynomial in $n$
variables with $\Theta\left(n^{2}\right)$ nonzero coefficients, then
$\Pr\left(f\left(\xi\right)=x\right)\le n^{o\left(1\right)-1/2}$.

The exponent of $1/2$ in Costello's result is best-possible, as can
be seen by considering the polynomial $\left(\xi_{1}+\dots+\xi_{n}\right)^{2}$.
However, it seems that for a ``typical'' quadratic polynomial
$f$ we should expect a much stronger bound. For example, if all
the coefficients of $f$ are integers of bounded size (and $\Theta\left(n^{2}\right)$
of them are non-zero), then the standard deviation of $f\left(\xi\right)$
is of order $\Theta\left(n\right)$. In this case it seems reasonable to assume
that ``typically'' the probability mass is roughly evenly distributed
over the integer points in a standard-deviation-sized interval around
the mean, yielding a bound of about $1/n$ for the point probabilities. Costello made a conjecture (see \cite[Conjecture~3]{Cos13}) to
this effect, suggesting that the only way $f\left(\xi\right)$ can
have point probabilities greater than $n^{\varepsilon-1}$, for any
constant $\varepsilon>0$, is if $f$ ``differs in only a few coefficients''
from a polynomial which splits into two linear factors. We remark
that Costello's paper was about polynomials with \emph{complex} coefficients,
and splitting over $\CC$ is a weaker property than splitting over
$\RR$. Nevertheless, Costello's conjecture appears to be plausible over
both $\RR$ and $\CC$.

To put Costello's conjecture in a wider context, an important goal for Littlewood--Offord-type problems is to prove \emph{inverse theorems}: in addition to proving general bounds on the maximum point probability, we would also like to understand the structural features exhibited whenever the random variable has a point probability close to this maximum. In the linear case, as previously mentioned, the point probabilities are only affected by the arithmetic structure of the multiset of coefficients $(a_{1},\dots,a_{n})$, and state-of-the-art inverse theorems due to Rudelson and Vershynin~\cite{RV08} and Nguyen and Vu~\cite{NV11} give a very refined understanding of the way this arithmetic structure is influenced by the maximum point probability (these results build on an earlier, coarser, inverse theorem due to Tao and Vu~\cite{TV09b}). However, in the quadratic case the point probabilities are influenced not only by the values of the coefficients, but also by how the different coefficients are arranged in the polynomial (for example, even the case where all the coefficients lie in $\{0,1\}$ is not well understood). Nguyen~\cite{Ngu12} proved a coarse inverse theorem (whose exact statement is too technical to reproduce here) showing that if, for a quadratic polynomial $f$, the maximum point probability of $f(\xi)$ is only polynomially small (that is, $\Pr(f(\xi)=x)\ge n^{-O(1)}$ for some $x$), then $f$ enjoys some algebraic and arithmetic structure. One can interpret Costello's conjecture as asking for a much more refined inverse theorem, albeit one that only takes algebraic structure into account.

In this paper, we prove some inverse theorems of a similar flavour to Costello's
conjecture, giving a connection between anti-concentration of $f\left(\xi\right)$ and algebraic properties of $f$. Roughly speaking, we prove that
if $f$ has concentration probability much larger than $1/n$ then
it must be close to a quadratic form with low rank\footnote{Recall that an $n$-variable \emph{quadratic form} over a field $\FF$ is a homogeneous quadratic polynomial $h\in \FF[x_1,\dots,x_n]$. If $\FF$ has characteristic not equal to 2, there is a unique representation $h(x)=x^T Q x$ with a symmetric matrix $Q\in \FF^{n\times n}$ (where $x$ denotes the column vector with entries $x_1,\dots,x_n$). The \emph{rank} of $h$ is defined to be the rank of this matrix $Q$. Equivalently, the rank of $h$ is the minimum $r$ such that there is a representation $h=\lambda_1 h_1^2+\dots+\lambda_r h_r^2$ as a linear combination of squares of homogeneous linear polynomials $h_1,\dots,h_r\in \FF[x_1,\dots,x_n]$.}.
Our first result is in terms of ``coefficient $L_{1}$ distance''.

\begin{thm}\label{thm:rank-L1}
Let $\FF\in \lbrace \CC,\RR, \QQ\rbrace$. For any integer $r\geq 3$, and any $0< \varepsilon\leq 1$, there is a constant $C=C(r,\varepsilon)$ such that the following holds. Let $f\in \FF[x_1,\dots,x_n]$ be a quadratic polynomial all of whose coefficients have absolute value at most $1$. Let $\xi=(\xi_1,\dots,\xi_n)\in \Rad^n$, and suppose that we have
\[\sup_{x\in \FF}\Pr\left(f(\xi)=x\right)\geq C\cdot \frac{\left(\log n\right)^{r/2}}{n^{1-2/\left(r+2\right)}}.\]
Then there is a quadratic form $h\in\FF\left[x_{1},\dots,x_{n}\right]$
of rank strictly less than $r$ such that the sum of the absolute values of the coefficients of $f-h$ is at most  $\varepsilon n^2$.
\end{thm}

Note that if $r$ is large, then the bound on the point probabilities in \cref{thm:rank-L1} is close to $1/n$. Also, note that we can rescale any quadratic polynomial $f\in \FF[x_1,\dots,x_n]$ so that all of its coefficients have absolute value at most 1, so \cref{thm:rank-L1} can be interpreted as giving a bound in terms of the largest coefficient of $f$. We remark that for the linear Littlewood--Offord problem, it essentially suffices to consider the case where the coefficients are
of bounded size, because if there are many coefficients with dramatically
different orders of magnitude, the point probabilities are small for
trivial reasons (Littlewood and Offord's original work~\cite{LO43}
proceeded along these lines). Unfortunately we were not able to find
such a reduction for the quadratic Littlewood--Offord problem.

In certain combinatorial applications, the coefficients of $f$ are integers of bounded size, or lie in some other bounded set of ``allowed coefficients''. For such polynomials, our next result gives a bound analogous to \cref{thm:rank-L1} where the quadratic form $h$ differs in only few coefficients from $f$ (so this version is in terms of ``coefficient Hamming distance'' as in Costello's conjecture, instead of ``coefficient $L_1$ distance'' as in \cref{thm:rank-L1}).

\begin{thm}
\label{thm:rank-edit}Let $\FF\in \lbrace \CC,\RR, \QQ\rbrace$. For any integer $r\geq 3$, any $0< \varepsilon\leq 1$, and any finite set $S\su \FF$, there is a constant $C=C(r,\varepsilon, S)$ such that the following holds. Let $f\in \FF[x_1,\dots,x_n]$ be a quadratic polynomial all of whose degree-2 coefficients are elements of the set $S$. Let $\xi=(\xi_1,\dots,\xi_n)\in \Rad^n$, and suppose that we have
\[\sup_{x\in \FF}\Pr\left(f(\xi)=x\right)\geq C\cdot \frac{\left(\log n\right)^{r/2}}{n^{1-2/\left(r+2\right)}}.\]
Then there is a quadratic form $h\in\FF\left[x_{1},\dots,x_{n}\right]$
of rank strictly less than $r$ such that $f$ and $h$ differ in at most $\varepsilon n^2$ coefficients.
\end{thm}

We remark that \cref{thm:rank-L1,thm:rank-edit} can in fact be used to
give anti-concentration estimates substantially stronger than $1/\sqrt{n}$ for
polynomials that are not close to factorising over $\CC$, as in Costello's
conjecture. If a complex quadratic form has rank at most 2, then it is a sum of two squares of linear forms. Over $\CC$, such forms
always split into linear factors. Therefore, applying \cref{thm:rank-L1,thm:rank-edit}
with $r=3$, we see that if the point probabilities of $f\left(\xi\right)$
are much larger than $n^{-3/5}$, then there is a quadratic form $h$,
close to $f$, which splits into linear factors over the complex numbers.

The proofs of \cref{thm:rank-L1,thm:rank-edit} involve a number of ideas and ingredients that may be independently interesting. In \cref{sec:outline} we will outline the proofs and discuss these ideas, but first we describe an application of \cref{thm:rank-edit} to anti-concentration of edge-statistics in Ramsey graphs, asymptotically answering a question of Kwan, Sudakov and Tran~\cite{KST}.

\subsection{Ramsey graphs}

An induced subgraph of a graph is said to be \emph{homogeneous} if
it is a clique or independent set. A classical result in Ramsey theory,
proved in 1935 by Erd\H os and Szekeres~\cite{ES35}, is that every
$n$-vertex graph has a homogeneous subgraph with at least $\frac{1}{2}\log_{2}n$
vertices. On the other hand, Erd\H os~\cite{Erd47} famously used
the probabilistic method to prove that, for all $n$, there exists
an $n$-vertex graph with no homogeneous subgraph on $2\log_{2}n$
vertices. Despite significant effort (see for example \cite{BRSW12,CZ16,Coh16,FW81,Li18}),
there are no known non-probabilistic constructions of graphs whose largest homogeneous subgraphs are of a comparable size.

Say an $n$-vertex graph is \emph{$C$-Ramsey} if it has no homogeneous
subgraph of size $C\log_{2}n$. It is widely believed that for any fixed constant $C$ all $C$-Ramsey graphs must in some sense resemble random graphs, and this belief has been supported by a number of theorems showing that certain ``richness'' properties characteristic of random graphs hold for all $C$-Ramsey graphs. The first result of this type was due to Erd\H os and Szemer\'edi~\cite{ES72}, who showed that every $C$-Ramsey graph $G$ has edge-density bounded away from zero and one. Note that this implies fairly strong information about the edge distribution on induced subgraphs of $G$, because any $n^{\alpha}$-vertex induced subgraph of an $n$-vertex $C$-Ramsey graph is itself $(C/\alpha)$-Ramsey.

This basic result was the foundation for a large amount of further
research on Ramsey graphs; over the years many conjectures have been proposed and resolved (see \cite{ABKS09,AK09,AKS03,BS07,Erd92,Erd97,EH77,KS17,KS17a,NST16,PR99,She98}).
In particular, we mention two results regarding the edge distribution
on induced subgraphs. First, solving a conjecture of Narayanan, Sahasrabudhe
and Tomon~\cite{NST16} (inspired by an old question of Erd\H os
and McKay~\cite{Erd92,Erd97}), Kwan and Sudakov~\cite{KS17a} proved
that for any $n$-vertex $C$-Ramsey graph there are
induced subgraphs with $\Omega\left(n^{2}\right)$ different numbers
of edges. Second, resolving a conjecture of Erd\H os, Faudree and
S\'os~\cite{Erd92,Erd97} (improving results of Alon and Kostochka~\cite{AK09}
and Alon, Balogh, Kostochka and Samotij~\cite{ABKS09}), Kwan and
Sudakov~\cite{KS17} also proved that every $n$-vertex $C$-Ramsey
graph has the property that for $\Omega\left(n\right)$ of the choices
$\ell\in\left\{ 0,\dots,n\right\} $, there are $\ell$-vertex induced
subgraphs with $\Omega\left(n^{3/2}\right)$ different numbers of
edges.

The aforementioned Erd\H os--Szemer\'edi theorem can be interpreted
as a (weak) ``concentration'' theorem: the numbers of edges in induced
subgraphs cannot be ``too extreme''. On the other hand, the two
results in the last paragraph point in the opposite direction: there
are many different possibilities for the numbers of edges in induced
subgraphs. In connection with some recent work on anti-concentration
of edge-statistics (see \cite{AHKT,FS,KST,MMNT}), Kwan, Sudakov and
Tran~\cite{KST} asked about anti-concentration of the edge distribution
in Ramsey graphs. Specifically, for an $n$-vertex $C$-Ramsey graph, let $X$ be the number of edges induced by a uniformly random set of (say) $n/2$ vertices. Is it true that $\Pr\left(X=x\right)=O\left(1/n\right)$
for all $x\in\NN$? If true, this would be best-possible, as can be
seen by considering a random graph $\GG\left(n,1/2\right)$. One of the motivations for this question was that better understanding of edge-statistics in Ramsey graphs could lead to a unified and more conceptual proof of the conjectures of Narayanan--Sahasrabudhe--Tomon and Erd\H os--Faudree--S\'os concerning induced subgraphs of Ramsey graphs with different numbers of edges. We discuss this further in \cref{sec:concluding}.

As an application of \cref{thm:rank-edit},
we answer Kwan, Sudakov and Tran's question asymptotically. Roughly speaking, we express
$X$ as a quadratic polynomial and show that Ramsey graphs are too
disordered for this polynomial to be close to a low-rank quadratic form.
\begin{thm}
\label{thm:ramsey-KST}The following holds for any fixed constants $C,c>0$. Let $G$ be an $n$-vertex $C$-Ramsey graph, and, for some $cn\le k\le\left(1-c\right)n$, let $X$ be the
number of edges induced by a uniformly random subset of $k$ vertices
of $G$. Then for any $x\in\ZZ$,
we have
\[
\Pr\left(X=x\right)\le n^{o\left(1\right)-1}.
\]
\end{thm}

In \cref{sec:concluding} we discuss a further related conjecture, and some connections between this line of research and some older conjectures about Ramsey graphs.

\subsection{Outline of the paper and the proofs}\label{sec:outline}

The structure of the paper is as follows. First, in \cref{sec:ramsey} we present the deduction of \cref{thm:ramsey-KST} (our result about Ramsey graphs) from \cref{thm:rank-edit}. This illustrates
some ideas that might be more generally useful in other applications
of \cref{thm:rank-L1,thm:rank-edit}. Via a coupling trick (\cref{lem:KST-coupling}), our random variable $X$ can be represented
as a certain quadratic polynomial. To apply \cref{thm:rank-edit} we then need to show
that a certain $n\times n$ matrix $M$ corresponding to this quadratic
polynomial (defined in terms of a Ramsey graph) is far from being
low-rank, in the sense that for any fixed $r\in\NN$, changing any
$o\left(n^{2}\right)$ entries of $M$ results in a matrix with rank
at least $r$. We observe that, to this end, it suffices to show that
for any $r\in\NN$, our matrix $M$ contains $\Omega\left(n^{2r}\right)$
invertible $r\times r$ submatrices. This will follow from a simple
generalisation (\cref{lem-many-copies}) of an old result due to Erd\H os and Hajnal: for fixed
$h$ and $C$, every $C$-Ramsey graph contains $\Omega\left(n^{h}\right)$
copies of every possible $h$-vertex induced subgraph.

Next, in \cref{sec:main-lemma} we state and prove an anti-concentration inequality for real quadratic
polynomials satisfying a certain technical non-degeneracy condition (\cref{lem:key}). This will be a key ingredient in the proofs of \cref{thm:rank-L1,thm:rank-edit}. The main idea that allows us to prove bounds stronger than $1/\sqrt{n}$ is
a decoupling trick applied to the characteristic function (Fourier
transform) of our random variable of interest. In circumstances where the characteristic function decays in a ``Gaussian-like'' way,
this allows us to reduce from the quadratic case to the linear case
without incurring the square-root loss that is usually associated
with decoupling tricks of this type.

In \cref{sec:complex-to-real} we state and prove a lemma concerning ``real projections'' of complex matrices (\cref{lem:L1-real-to-complex}), which essentially allows us to deduce the complex cases of \cref{thm:rank-L1,thm:rank-edit} from the real cases. This lemma may also be of independent interest. To illustrate, a special case is the fact that for any nonsingular complex matrix $A$, there is a phase
$\theta\in\left[-\pi,\pi\right]$ such that $\Re\left(e^{i\theta}A\right)$
is nonsingular.

Then, in \cref{sec:rank} we outline how to deduce \cref{thm:rank-L1,thm:rank-edit} from \cref{lem:key,lem:L1-real-to-complex}. The main step of the deduction is to show that quadratic polynomials which are far from
a low-rank quadratic form satisfy the technical non-degeneracy condition of \cref{lem:key}. As it happens, this step is more challenging than it may first seem; it basically amounts to proving that if a matrix is close to being symmetric, and close to having low rank, then it is close to a matrix
that is simultaneously symmetric and has low rank. This fact about ``symmetric low-rank approximation'' is encapsulated in \cref{lem-matrix-rank-close,lem-matrix-rank-close-set} (there are slightly different versions for the proofs of \cref{thm:rank-L1,thm:rank-edit}), and most of the rest of the paper is devoted to proving these lemmas. Indeed, in \cref{sec:aux} we prove some general-purpose lemmas about a certain notion of ``robust linear independence'', and in \cref{sec:lem-matrix-rank-close} we use these lemmas to prove \cref{lem-matrix-rank-close,lem-matrix-rank-close-set}.

Finally, \cref{sec:concluding} contains some concluding remarks. In particular, we
present a new conjecture about edge-statistics in Ramsey graphs, generalising
\cref{thm:ramsey-KST}, which would imply the conjectures of Erd\H os--Faudree--S\'os
and Narayanan--Sahasrabudhe--Tomon regarding subgraphs of Ramsey
graphs with different numbers of edges.

\subsection{Notation}

We use standard asymptotic notation throughout. For functions $f=f\left(n\right)$
and $g=g\left(n\right)$ we write $f=O\left(g\right)$ to mean there
is a constant $C$ such that $\left|f\right|\le C\left|g\right|$,
we write $f=\Omega\left(g\right)$ to mean there is a constant $c>0$
such that $f\ge c\left|g\right|$ for sufficiently large $n$, we
write $f=\Theta\left(g\right)$ to mean that $f=O\left(g\right)$
and $f=\Omega\left(g\right)$, and we write $f=o\left(g\right)$ or
$g=\omega\left(f\right)$ to mean that $f/g\to0$ as $n\to\infty$.
All asymptotics are as $n\to\infty$ unless specified otherwise.

For a non-negative integer $n$ we define $[n]=\lbrace 1,\dots, n\rbrace$, and for a real number $x$, the floor function is denoted $\floor x=\max\left\{ i\in\ZZ: i\le x\right\} $. For a vector $v\in \CC^n$ or a matrix $A\in \CC^{m\times n}$, we let $\Re(v)\in \RR^n$ and $\Re(A)\in \RR^{m\times n}$ denote the vector or matrix obtained by taking the real part of each entry. We adopt the convention that the determinant of the $0\times 0$ empty matrix is 1. All logarithms are in base $e$, unless explicitly noted otherwise.

We also use standard graph-theoretic notation. Given a graph $G$, we denote its vertex set by $V(G)$. For subsets $X,Y\su V(G)$, let $e(X)$ denote the number of edges inside $X$, and let $e(X,Y)$ denote the number of edges between $X$ and $Y$. Let $d(X)=e(X)/\binom{|X|}{2}$ and $d(X,Y)=e(X,Y)/(|X||Y|)$ denote the density of edges inside $X$, and between $X$ and $Y$, respectively. Abusing notation, we write $e(x,Y)$ or $e(x,y)$ to denote $e(\{x\},Y)$ or $e(\{x\},\{y\})$, respectively.

Given a field $\FF$ and non-negative integers $m$ and $n$, let $\FF^{m\times n}$ denote the set of all $m\times n$ matrices with entries in $\FF$. For a matrix $A\in \FF^{m\times n}$, for $i=1,\dots,m$ we write $\row_i(A)\in \FF^n$ for the vector corresponding to the $i$-th row of $A$, and for $j=1,\dots,n$ we write $\col_j(A)\in \FF^m$ for the vector corresponding to the $j$-th column. Given a matrix $A\in \FF^{m\times n}$ and a subset $I\su [n]$, let $A_I$ be the $m\times |I|$ submatrix of $A$ consisting of the columns with indices in $I$.  Similarly, for a vector $v\in \FF^n$ and a subset $I\su [n]$, let $v_I\in \FF^I$ be the vector consisting of the entries of $v$ with indices in $I$.

For a matrix $A\in \CC^{m\times n}$, we denote the sum of the absolute values of the entries of $A$ by $\Vert A\Vert_1$, and we denote the maximum of the absolute values of the entries by $\Vert A\Vert_{\infty}$ (these are entrywise norms of $A$, not to be confused with the more common operator norms). For a vector $v\in \CC^n$, we denote the usual $L_p$-norm by $\Vert v\Vert_p$, and for vectors $v,w\in \RR^n$ we write $\langle v,w\rangle$ for the standard inner product of $v$ and $w$.

\section{Anti-concentration in Ramsey graphs\label{sec:ramsey}}

In this section we deduce \cref{thm:ramsey-KST} from \cref{thm:rank-edit}. Our plan will be to express the random variable $X$ in \cref{thm:ramsey-KST} as a quadratic polynomial of independent Rademacher random variables. We can obtain an upper bound on the point probabilities of this polynomial from \cref{thm:rank-edit} if we can show that our quadratic polynomial is not close to a low-rank quadratic form. In order to do so, we will apply the following simple lemma to the matrix associated with the homogeneous degree-2 part of the polynomial.

\begin{lem}
\label{lem:many-full-rank}
Let $r$ be a positive integer, let $\delta\geq 0$, and let $M$ be an $m\times m$ matrix over any field which has more than $\delta m^{2r}$ full-rank $r\times r$ submatrices. Then, if we change up to $\delta m^2$ entries of $M$, the resulting matrix has rank at least $r$.
\end{lem}

\begin{proof}
If we change at most $\delta m^2$ entries, we can affect at most $\delta m^2 m^{2(r-1)}=\delta m^{2r}$ of the $r\times r$ submatrices of $M$, so a full-rank $r\times r$ submatrix remains.
\end{proof}

Next, the only fact about Ramsey graphs we will need is the fact that they have many copies of every possible induced subgraph on a small number of vertices. This generalises an old result of Erd\H{o}s and Hajnal~\cite{EH77}, which asserts the existence of at least one copy of each such subgraph.

\begin{lem}\label{lem-many-copies}
For any fixed $h\geq 1$ and any fixed constant $C>0$, there is $\delta=\delta(h,C)>0$ such that the following holds for sufficiently large $n$. Every $n$-vertex $C$-Ramsey graph $G$ contains at least $\delta n^h$ induced copies of every graph $H$ on $h$ vertices.
\end{lem}

\begin{proof}
A graph is said to be $\eps$-regular if for all subsets $X, Y\su V(G)$ with $\vert X\vert\geq \eps \vert V(G)\vert$ and $\vert Y\vert\geq \eps \vert V(G)\vert$, we have $|d(X,Y)-d(V(G))|\le\eps$. It is a consequence of Szemer\'edi's regularity lemma that for any fixed $\eps>0$, every $n$-vertex graph $G$ contains an $\eps$-regular induced subgraph $G[U]$ on $m=\Omega(n)$ vertices (see also \cite[Lemma~5.2]{CF12} for a version of this fact with better dependence on $\eps$). Now, if $G$ is $C$-Ramsey, then $G[U]$ is still $(C+o(1))$-Ramsey, so by the theorem of Erd\H{o}s and Szemer\'{e}di~\cite{ES72} mentioned in the introduction, the density of $G[U]$ is bounded away from zero and one: that is, there is $\eta>0$ depending only on $C$ such that $\eta\le d(U)\le 1-\eta$ for sufficiently large $n$.

Then, provided $\eps$ is sufficiently small with respect to $\eta$, we can conclude the proof by applying a counting lemma to $G[U]$ (see for example \cite[Lemma 5.12]{CF12}): if an $m$-vertex graph has density bounded away from zero and one, and it is $\eps$-regular for sufficiently small $\eps$, then it has $\Omega(m^h)$ induced copies of every graph $H$ on $h$ vertices. Thus, we obtain $\Omega(m^h)=\Omega(n^h)$ induced copies of $H$ in $G[U]$ and therefore in $G$.
\end{proof}

Another crucial ingredient is a variant of \cite[Lemma~2.8]{KST}, to express the random variable $X$ in \cref{thm:ramsey-KST} as a quadratic polynomial of independent random variables. Consider any graph $G$ with vertex set $[n]$ and any $0\leq k\leq n$, and let $m=\min(k,n-k)$. First, we want a way to generate a random $k$-vertex subset of $G$ in a way which involves independent random choices. Let $\pi$ be a uniformly random permutation of $[n]$ and let $\xi=(\xi_1,\dots,\xi_m)\sim\Rad^m$ be a sequence of independent Rademacher random variables (also independent from $\pi$). Note that
\begin{equation}\label{eq:U-pi-xi}
U_{\pi,\xi}=\lbrace \pi(i): i\in [m], \xi_i=1\rbrace\cup \lbrace \pi(i+m): i\in [m], \xi_i=-1\rbrace\cup \lbrace \pi(i): 2m+1\leq i\leq m+k\rbrace
\end{equation}
is a uniformly random subset of $k$ vertices of $G$. Indeed, the union of the first two sets on the right-hand side has size $m$, the third set has size $k-m\geq 0$, and all three sets are disjoint.

Now, recall that for two vertices $v,w\in V(G)$, we defined $e(v,w)=1$ if there is an edge between $v$ and $w$ and $e(v,w)=0$ otherwise. For any fixed outcome of $\pi$, the number of edges in $U=U_{\pi,\xi}$ is
\begin{equation}\label{eq:number-edges}
e(U_{\pi,\xi})=\sum_{1\leq i<j\leq n}\mathbf{1}_{\pi(i)\in U}\mathbf{1}_{\pi(j)\in U}e(\pi(i),\pi(j)).
\end{equation}
Note that by the definition of $U=U_{\pi,\xi}$, we have
\[\mathbf{1}_{\pi(i)\in U}=
\begin{cases}
\frac{1}{2}(1+\xi_i)&\text{if }1\leq i\leq m\\
\frac{1}{2}(1-\xi_{i-m})&\text{if }m+1\leq i\leq 2m\\
1&\text{if }2m+1\leq i\leq m+\ell\\
0&\text{if }m+\ell+1\leq i\leq n\\
\end{cases}.\]
Plugging this into \cref{eq:number-edges}, and using that $\xi_i^2=1$ for all $i$, we obtain the following lemma.

\begin{lem}
\label{lem:KST-coupling}
Let $G$ be a graph with vertex set $[n]$. Furthermore, let $0\leq k\leq n$ and $m=\min(k,n-k)$. Let $\pi$ be a random permutation of $[n]$, and $\xi=(\xi_1,\dots,\xi_m)\sim\Rad^m$ be a sequence of independent Rademacher random variables, and define $U_{\pi,\xi}\su V(G)$ as in \cref{eq:U-pi-xi}. Then $U_{\pi,\xi}$ is a uniformly random  subset of $k$ vertices of $G$. Furthermore we can write 
\[e(U_{\pi,\xi})=f_\pi(\xi)=\sum_{1\leq i<j\leq m}a_{ij}\xi_i\xi_j+\sum_{1\leq i\leq m}a_i\xi_i+a_0,\]
where the coefficients $a_{ij}$, $a_i$ and $a_0$ of the quadratic polynomial $f_\pi$ only depend on $\pi$ (and not on $\xi$). In addition, for $1\le i<j\le m$ we have
\begin{equation}\label{eq-value-a-ij}
a_{ij}=\frac{1}{4}e(\pi(i),\pi(j))-\frac{1}{4}e(\pi(i),\pi(j+m))-\frac{1}{4}e(\pi(i+m),\pi(j))+\frac{1}{4}e(\pi(i+m),\pi(j+m)).
\end{equation}
\end{lem}

Note that \cref{eq-value-a-ij} in particular implies that $a_{ij}\in \lbrace -\frac{1}{2}, -\frac{1}{4}, 0, \frac{1}{4}, \frac{1}{2}\rbrace$ for all $i<j$. One can also give explicit formulas for the other coefficients $a_i$ and $a_0$ of $f_\pi$, but this is not necessary for our argument.

Now we put everything together to prove \cref{thm:ramsey-KST}.

\begin{proof}
[Proof of \cref{thm:ramsey-KST}]
Fix some $r\ge 3$, which we treat as a constant in all asymptotic notation. We will prove that $\Pr(X=x)\le n^{-1+2/(r+2)+o(1)}$. Since $r$ was arbitrary, this suffices to prove \cref{thm:ramsey-KST}.

Let $G$ be an $n$-vertex $C$-Ramsey graph, and let $cn\le k\le (1-c)n$. As before, define $m=\min(k,n-k)$ and note that $cn\leq m\leq n/2$. As in \cref{lem:KST-coupling}, we can model the random variable $X$ as $X=e(U_{\pi,\xi})=f_\pi(\xi)$, where $\pi$ is a random permutation of $[n]$, and $\xi=(\xi_1,\dots,\xi_m)\sim\Rad^m$ is a sequence of independent Rademacher random variables.

Let us say that a $(2r)$-tuple $(i_1,\dots,i_r,j_1,\dots,j_r)\in [m]^{2r}$ is \emph{strong} if $i_1<\dots<i_r<j_1<\dots<j_r$ and if we have $a_{i_\ell j_\ell}=1/2$ for $\ell=1,\dots,r$, but $a_{i_\ell j_q}=0$ whenever $\ell\neq q$ (note that this definition depends on the permutation $\pi$; recall \cref{eq-value-a-ij}). We first use \cref{lem-many-copies} to show that there are likely to be many strong $(2r)$-tuples.

\begin{claim}\label{claim-f-pi-tuples}
Subject to the randomness of the random permutation $\pi$, with probability $1-e^{-\Omega(n)}$ there are $\Omega(m^{2r})$ strong $(2r)$-tuples.
\end{claim}

\begin{proof}[Proof of \cref{claim-f-pi-tuples}]
Since $G$ is a $C$-Ramsey graph, by \cref{lem-many-copies}, it has $\Omega(n^{4r})$ induced copies of a perfect matching on $4r$ vertices (consisting of $2r$ edges). That is to say, there are $\Omega(n^{4r})$ sequences of distinct vertices $(v_1,\dots,v_{4r})\in V(G)^{4r}$ such that for $i=1,\dots,2r$ there is an edge between $v_i$ and $v_{i+2r}$ and there are no other edges between $v_1,\dots,v_{4r}$.

There are $\binom{m}{2r}=\Omega(m^{2r})$ different $(2r)$-tuples $(i_1,\dots,i_r,j_1,\dots,j_r)\in [m]^{2r}$ with $i_1<\dots<i_r<j_1<\dots<j_r$. For each such $(2r)$-tuple,
\[(\pi(i_1), \dots, \pi(i_r),\pi(i_1+m),\dots ,\pi(i_r+m),\pi(j_1), \dots, \pi(j_r),\pi(j_1+m),\dots ,\pi(j_r+m))\]
is a uniformly random sequence of $4r$ distinct vertices of $G$. Thus, with probability $\Omega(1)$, it is one of the sequences $(v_1,\dots,v_{4r})$ considered above. But if that is the case, then $(i_1,\dots,i_r,j_1,\dots,j_r)$ is strong: by \cref{eq-value-a-ij}, $a_{i_\ell j_\ell}=\frac{1}{4}-0-0+\frac{1}{4}=\frac{1}{2}$ for $\ell=1,\dots,r$, but $a_{i_\ell j_q}=0-0-0+0=0$ whenever $\ell\neq q$.

Let $Z$ be the number of strong $(2r)$-tuples: we have just proved that $\E Z=\Omega(m^{2r})$. Now we can conclude the proof with a concentration
inequality. Note that changing $\pi$ by a transposition (swapping some $\pi(t)$ and $\pi(t')$) changes the number of strong $(2r)$-tuples by at most $8r\cdot m^{2r-1}$. Indeed, there are at most $8r\cdot m^{2r-1}$ different $(2r)$-tuples $(i_1,\dots,i_r,j_1,\dots,j_r)\in [m]^{2r}$ such that $t$ or $t'$ occur among $i_1,\dots,i_r,i_1+m,\dots, i_r+m,j_1,\dots,j_r, j_1+m,\dots, j_r+m$ (which are the only places where the value of $\pi$ affects whether $(i_1,\dots,i_r,j_1,\dots,j_r)$ is strong). Thus, by a McDiarmid-type concentration inequality for random permutations (see for example \cite[Section 3.2]{McD98}), we have
\[\Pr(Z< \E Z/2)\le \exp\left(-\Omega\left(\frac{(\E Z/2)^2}{n\cdot (8r\cdot m^{2r-1})^2}\right)\right)=\exp\left(-\Omega\left(\frac{(m^{2r})^2}{n\cdot (8r\cdot m^{2r-1})^2}\right)\right)=\exp(-\Omega(n)),\]
recalling that $m\geq cn$.
\end{proof}

Now, condition on an outcome of $\pi$ satisfying the conclusion of \cref{claim-f-pi-tuples}. Conditionally, $X$ can be represented as a quadratic polynomial $f_{\pi}(\xi)$ in $\xi=(\xi_1,\dots,\xi_m)\sim\Rad^m$. We can express the homogeneous degree-2 part of $f_\pi(\xi)$ as $\xi^T Q_{\pi} \xi$ for some symmetric $m\times m$ matrix, and note that for $i< j$ the $(i,j)$-entry of $Q_{\pi}$ equals $a_{ij}/2$. \cref{claim-f-pi-tuples} implies that the matrix $Q_\pi$ has $\Omega(m^{2r})$ full-rank $r\times r$ submatrices (note that for every strong $(2r)$-tuple $(i_1,\dots,i_r,j_1,\dots,j_r)$ the submatrix with rows $i_1,\dots,i_r$ and columns $j_1,\dots,j_r$ is a diagonal matrix with entries $1/4$ on the diagonal). Hence by \cref{lem:many-full-rank} there is $\eps=\Omega(1)$ such that whenever we change up to $2\eps m^2$ entries of $Q_{\pi}$, the resulting matrix has rank at least $r$. Now, if $h(\xi)$ is a quadratic form differing from $f$ in at most $\eps m^2$ coefficients, then $h$ is of the form $\xi^T Q_{\pi}' \xi$ for a symmetric $m\times m$ matrix $Q_\pi'$ which differs from $Q_{\pi}$ in at most $2\eps m^2$ entries and consequently has rank at least $r$. Thus, using that the degree-2 coefficients $a_{ij}$ of $f_{\pi}$ all lie in the set $S=\lbrace -\frac{1}{2}, -\frac{1}{4}, 0, \frac{1}{4}, \frac{1}{2}\rbrace$, applying \cref{thm:rank-edit} yields
\[\sup_{x\in \QQ}\Pr\left(f_\pi(\xi)=x\right)<C(r,\eps, S)\cdot \frac{\left(\log n\right)^{r/2}}{(cn)^{1-2/\left(r+2\right)}}\le \frac{n^{o(1)}}{n^{1-2/\left(r+2\right)}}.\]
Recalling that we have been conditioning on an event that holds with probability $1-e^{-\Omega(n)}$, it follows that
\[\sup_{x\in \QQ}\Pr\left(X=x\right)=\frac{n^{o(1)}}{n^{1-2/\left(r+2\right)}}+e^{-\Omega(n)}=n^{-1+2/(r+2)+o(1)},\]
as desired.
\end{proof}

\section{A technical anti-concentration inequality for real polynomials\label{sec:main-lemma}}

In this section, we will prove an anti-concentration bound for real quadratic polynomials satisfying a certain technical non-degeneracy condition. This will be one of the key ingredients for the proofs of \cref{thm:rank-L1,thm:rank-edit}.

To cleanly state our anti-concentration inequality, we first make some simple definitions.

\begin{defn}\label{def:non-degenerate}
For an $n\times n$ matrix $M$ and a tuple $(i_{1},\dots,i_{r})\in[n]^r$, let $M(i_{1},\dots,i_{r})$ be the $r\times n$ matrix whose rows are $\row_{i_1}(M),\dots,\row_{i_r}(M)$. For $\delta>0$, let us say that a $r\times n$ matrix $M'$ is \emph{$\delta$-non-degenerate} if for any unit vector $e\in\RR^{r}$, there are at least $\delta n$ columns $w$ of $M'$ satisfying $\left|\left\langle w,e\right\rangle \right|\ge\delta$.
\end{defn}

Now, our anti-concentration inequality is as follows.

\begin{lem}
\label{lem:key}For any integer $r\geq 3$ and any $\delta> 0$ there is a constant $C=C\left(r,\delta\right)$
such that the following holds. Consider a real quadratic polynomial $f\left(x\right)=\sum_{1\le i\le j\le n}a_{ij}x_{i}x_{j}+\sum_{1\le i\le n}a_{i}x_{i}+a_0$, let $a_{ji}=a_{ij}$ for $i>j$, and let $M=(a_{ij})_{i,j}\in \RR^{n\times n}$. Suppose that each $|a_{ij}|\le 1$, and suppose
that there is a set $T\subseteq[n]^{r}$ of $\delta n$
disjoint $r$-tuples such that for each $(i_{1},\dots,i_{r})\in T$, the
matrix $M(i_{1},\dots,i_{r})$ is $\delta$-non-degenerate. Then, for $\xi\in\Rad^{n}$ and any $x\in\RR$, we have
\[
\Pr\left(\left|f\left(\xi\right)-x\right|\le n^{2/\left(r+2\right)}\right)\le C\cdot\frac{\left(\log n\right)^{r/2}}{n^{1-2/\left(r+2\right)}}.
\]
\end{lem}

The notion of being $\delta$-non-degenerate is closely related to the condition in an anti-concentration theorem due to Hal\'asz~\cite{Hal77} (stated as \cref{thm:halasz} below), which will be used in the proof of \cref{lem:key}. A matrix $M'$ being $\delta$-non-degenerate can be interpreted as a robust version of $M'$ having full row rank, so if an $n\times n$ matrix $M$ has many $r\times n$ submatrices that are $\delta$-non-degenerate, then there is some sense in which $M$ robustly has rank at least $r$. The reader may wish to compare the statement of \cref{lem:key} with the statements of \cref{thm:rank-L1,thm:rank-edit} in the introduction.

Before proving \cref{lem:key}, we discuss some of the main ideas and ingredients. The most crucial idea is a variant of a \emph{decoupling} trick due to Costello, Tao and Vu~\cite{CTV06}, as follows. If $[n]=I\cup J$ is a partition of the index set into two subsets, then we can break $\xi=(\xi_1,\dots,\xi_n)$ into two subsequences $\xi_{I}$ and $\xi_{J}$. The quadratic polynomial $f(\xi)$ can then be written as $f\left(\xi\right)=f\left(\xi_{I},\xi_{J}\right)$.
The crucial observation is that if $\xi_{J}'$ is an independent copy of $\xi_{J}$
then it is possible to relate the anti-concentration of $f(\xi)$
to the anti-concentration of $Y:=f(\xi_{I},\xi_{J})-f(\xi_{I},\xi_{J}')$:
for example, one can use the Cauchy--Schwarz inequality to prove
that \begin{equation}\Pr(f(\xi)=x)\le\Pr\big( f\left(\xi_{I},\xi_{J}\right)=x\text{ and }f\left(\xi_{I},\xi_{J}'\big)=x\right)^{1/2}\le \Pr(Y=0)^{1/2}.\label{eq:CTV}\end{equation}
After conditioning on an outcome of $(\xi_{J},\xi_{J}')$,
the random variable $Y$ then becomes a \emph{linear} polynomial in
$\xi_{I}$, which is much easier to study.

This approach results in a square-root loss, and therefore seems unsuitable to prove \cref{lem:key}. However, a variation on this approach is to instead
use decoupling to study the modulus of the \emph{characteristic function} (Fourier
transform) $t\mapsto \E e^{2\pi i t f(\xi)}$ of $f(\xi)$. Specifically, we will use the following
simple observation.

\begin{lem}
\label{lem:complex-decoupling}Let $\xi_I$ and $\xi_J$ be independent random vectors, and let $f(\xi_I,\xi_J)$ be a real-valued random variable defined in terms of these random vectors. Let $\xi_J'$ be an independent copy of $\xi_J$. Then for any $t\in \RR$,
$$
\left|\E e^{2\pi itf(\xi_I,\xi_J)}\right|^{2}  \le \E\left[\left|\E[e^{2\pi it\left(f(\xi_{I},\xi_{J})-f(\xi_{I},\xi_{J}')\right)}\mid \xi_{J},\xi_{J'}]\right|\right].
$$
\end{lem}

\begin{proof}
First, by convexity we have
$$
\left|\E e^{2\pi itf(\xi_I,\xi_J)}\right|^{2} = \left|\E\left[\E [e^{2\pi itf(\xi_I,\xi_J)} | \xi_{I}]\right]\right|^{2} \le\E\left[\left|\E[e^{2\pi it f\left(\xi_{I},\xi_{J}\right)} | \xi_{I}]\right|^{2}\right].
$$
Then, observe that for independent identically distributed complex-valued random variables $Z,Z'$, we have 
$$
\left|\E Z\right|^{2}=\E Z\overline{\E Z}=\E Z\overline{\E Z'}=\E\left[Z\overline{Z'}\right].
$$
In particular, we obtain
$$
\left|\E[e^{2\pi it f\left(\xi_{I},\xi_{J}\right)}\mid \xi_{I}]\right|^{2}=\E[e^{2\pi it\left(f(\xi_{I},\xi_{J})-f(\xi_{I},\xi_{J}')\right)}\mid\xi_{I}].
$$
It follows that
$$
\left|\E e^{2\pi itf(\xi_I,\xi_J)}\right|^{2} \le \E\left[\E[e^{2\pi it\left(f(\xi_{I},\xi_{J})-f(\xi_{I},\xi_{J}')\right)}\mid\xi_{I}]\right]
 =\E\left[\E[e^{2\pi it\left(f(\xi_{I},\xi_{J})-f(\xi_{I},\xi_{J}')\right)}\mid\xi_{J},\xi_{J'}]\right],
$$
from which we can conclude the desired result.
\end{proof}

We remark that while we were working on this paper, some similar decoupling
tricks were independently developed by Berkowitz~\cite{Ber18} in
connection with his work on local central limit theorems for clique counts in random
graphs. We also remark that a very similar argument appears implicitly in a paper of Nguyen~\cite{Ngu12}.

Next, the following result is called \emph{Ess\'een's concentration inequality}~\cite{Ess66}. It gives a way to prove anti-concentration bounds by integrating bounds on the characteristic function of a random variable. This particular statement is a special case of \cite[Lemma~7.17]{TV}.

\begin{lem}[\cite{TV}]
\label{lem:esseen}There is a constant $C>0$ such that the following holds. Let $X$ be a real-valued random variable which takes only a finite number of values. Then for any $\varepsilon>0$, any $s>0$ and any $x\in \RR$, we have
\[ \Pr\left(\left|X-x\right|\le s\right)\le C\left(s+1/\varepsilon\right)\int_{-\varepsilon}^{\varepsilon}\left|\E e^{2\pi itX}\right|\d t.\]
\end{lem}

It may not be immediately obvious how one can benefit from using a decoupling trick for characteristic functions (as in \cref{lem:complex-decoupling}) instead of using a decoupling trick for point probabilities directly (as in \cref{eq:CTV}). Indeed, \cref{lem:complex-decoupling} also involves a square-root loss when studying $f(\xi_I,\xi_J)$ via $f(\xi_{I},\xi_{J})-f(\xi_{I},\xi_{J}')$. The key is that the square-root loss is ``inside the integral''. It turns out that in the setting of \cref{lem:key}, the characteristic function has ``sharp threshold'' behaviour: if $|t|$ is much smaller than $1/n$, then $|\E e^{2\pi i t f(\xi)}|$ is very close to one, whereas if $|t|$ is much larger than $1/n$ then $|\E e^{2\pi i t f(\xi)}|$ is very close to zero. Therefore taking the square root of the modulus of the characteristic function has a relatively small effect on its integral.

In order to effectively apply \cref{lem:complex-decoupling}, we will need some understanding of the typical structure of $f(\xi_{I},\xi_{J})-f(\xi_{I},\xi_{J}')$ as a linear polynomial in $\xi_{I}$, subject to the randomness of $\xi_{J}$ and $\xi_{J}'$. In particular, we need to show that this polynomial is unlikely to have many coefficients that are close to zero. To accomplish this, we use 
the following multi-dimensional extension of the (linear) Littlewood--Offord theorem
due to Hal\'asz~\cite{Hal77}\footnote{We remark that a very similar inequality was also proved by Tao and
Vu~\cite[Theorem~1.4]{TV12}, and that there is a large body of work
generalising the Erd\H os--Littlewood--Offord theorem to higher
dimensions without this non-degeneracy condition (in which case the
bounds are much weaker; see for example the survey \cite[Section~2]{NV13}).}.

\begin{thm}
\label{thm:halasz}For any integer $d\geq 1$ and any $\delta>0$, there is $C=C(d,\delta)>0$ such that the following holds. Let $a_{1},\dots,a_{n}$ be a collection of vectors in $\RR^{d}$ and let $s>0$. Suppose that for any unit vector $e\in\RR^{d}$, there are at least $\delta n$ vectors $a_{i}$ with $|\langle a_{i},e\rangle|\ge s$.
Then for $\xi=(\xi_1,\dots,\xi_n)\in\Rad^{n}$ we have
\[
\sup_{u\in\RR^{d}}\Pr\left(\left\Vert \sum_{i=1}^{n}\xi_{i}a_{i}-u\right\Vert _{2}<s\right)\le Cn^{-d/2}.
\]
\end{thm}

We have still not yet described how to choose the partition $[n]=I\cup J$ for decoupling. We will actually just choose the partition randomly; we will then need the fact that a random submatrix of a non-degenerate
matrix is typically still non-degenerate, as follows. Recall that for a matrix $M\in \RR^{r\times n}$ and a subset $I\su [n]$, we defined $M_I$ to be the $r\times |I|$ submatrix of $M$ consisting of the columns with indices in $I$. Also recall that matrix norms in this paper are entrywise.

\begin{lem}
\label{lem:inherit-non-degeneracy}Fix an integer $r\geq 1$ and fix $\delta>0$. Suppose $M\in\RR^{r\times n}$ is a $\delta$-non-degenerate matrix with $\Vert M\Vert_\infty\le 1$. Then, for a uniformly random subset $I\su [n]$, with probability $1-e^{-\Omega\left(n\right)}$,
the matrix $M_{I}$ is $\left(\delta/3\right)$-non-degenerate. (Here, the implicit constant in the $\Omega$-term may depend on $r$ and $\delta$.)
\end{lem}

\begin{proof}
Let $\eps=\delta/(2r)>0$, and fix a finite set $E\su \lbrace e\in \RR^r: \Vert e\Vert_2=1\rbrace$ such that for any unit vector $e\in \RR^r$ we can find $e'\in E$ with $\Vert e-e'\Vert_2\leq \eps$ (that is, $E$ is an $\eps$-net of the unit sphere in $\RR^r$).

Note that every column $w\in \RR^r$ of $M$ satisfies $\Vert w\Vert_2\leq r$. Thus, whenever $e'\in E$ and $w$ is a column of $M$ with $\left|\left\langle w,e'\right\rangle \right|\ge\delta$, then all unit vectors $e\in \RR^r$ with $\Vert e-e'\Vert_2\leq \eps$ satisfy
\[\left|\left\langle w,e\right\rangle \right|\geq 
\left|\left\langle w,e'\right\rangle \right|-\left|\left\langle w,e'-e\right\rangle \right|\geq
\delta-\Vert w\Vert_2 \Vert e-e'\Vert_2\geq \delta - r\eps\geq \delta/3.\]
For every $e'\in E$, the $\delta$-non-degenerate matrix $M$ has at least $\delta n$ columns $w$ such that $\left|\left\langle w, e'\right\rangle \right|\ge\delta$. By a Chernoff bound, with probability $1-e^{-\Omega\left(n\right)}$, at least $\left(\delta/3\right)n$ of these columns are still in $M_{I}$. By taking the union bound over all $e'\in E$, we see that with probability $1-e^{-\Omega\left(n\right)}$, for every $e'\in E$ the matrix $M_I$ has $(\delta/3) n$ columns $w$ with $\left|\left\langle w,e'\right\rangle \right|\ge\delta$. Whenever this happens, for every unit vector $e\in \RR^r$, the matrix $M_I$ has $(\delta/3) n$ columns $w$ with $\left|\left\langle w,e\right\rangle \right|\ge\delta/3$. As $M_I$ has at most $n$ columns in total, this means that $M_I$ is $(\delta/3)$-non-degenerate.
\end{proof}

We are now ready to prove \cref{lem:key}.

\begin{proof}
[Proof of \cref{lem:key}]
Fix an integer $r\geq 3$ and fix $\delta> 0$. For all asymptotic notation in this proof, we treat $r$ and $\delta$ as constants. Let $f\in \RR[x_1,\dots,x_n]$ and $T\subseteq [n]^r$ be as in the lemma statement. For any subset $J \subseteq [n]$ and tuple
$(i_{1},\dots,i_{r})\in T$, let $M_{J}(i_{1},\dots,i_{r})$ be
the submatrix of $M(i_{1},\dots,i_{r})$ consisting of the columns with indices $j\in J$.

Now, consider a uniformly random subset $J\subseteq [n]$, and let $I=[n]\setminus J$. By \cref{lem:inherit-non-degeneracy} and the union bound, with probability $1-e^{-\Omega(n)}$ each $M_{J}(i_{1},\dots,i_{r})$, for $(i_1,\dots,i_r)\in T$, is $(\delta/3)$-non-degenerate. Also, by a Chernoff bound, with probability $1-e^{-\Omega(n)}$ we have $\vert J\vert\geq n/3$ and $\vert I^r\cap T\vert\geq 2^{-r-1}\vert T\vert\geq 2^{-r-1}\delta n$.

Thus, we can fix a partition $[n]=I\cup J$ such that $\vert J\vert\geq n/3$ and such that there exists a set $T_{I}\subseteq  I^{r}\cap T$ of at least $2^{-r-1}\delta n=\Omega(n)$ disjoint $r$-tuples, with the property that for each $(i_{1},\dots,i_{r})\in T_{I}$, the matrix $M_{J}(i_{1},\dots,i_{r})$ is $(\delta/3)$-non-degenerate.

Now, let $\xi_{I}=\left(\xi_{\ell}\right)_{\ell\in I}$ and $\xi_{J}=\left(\xi_{j}\right)_{j\in J}$. We write $f\left(\xi\right)=f(\xi_I, \xi_J)$. Let $\xi_{J}'$ be an independent copy of $\xi_{J}$, so by \cref{lem:complex-decoupling}, for any $t\in \RR$ we have
\begin{equation}\label{eq:decoupled}
\left|\E e^{2\pi itf(\xi)}\right|^{2}\le\E\left[\left|\E[e^{2\pi it\left(f(\xi_{I},\xi_{J})-f(\xi_{I},\xi_{J}')\right)}\mid \xi_{J},\xi_{J'}]\right|\right].
\end{equation}
Note that $f(\xi_{I},\xi_{J})-f(\xi_{I},\xi_{J}')=\sum_{\ell\in I}A_{\ell}\xi_{\ell}+A$, where $A_{\ell}=\sum_{j\in J}a_{\ell j}\left(\xi_{j}-\xi_{j}'\right)$ for all $\ell\in I$ and
\[
A=\sum_{j,j'\in J,\,j\le j'}a_{jj'}\left(\xi_j\xi_{j'}-\xi_j'\xi_{j'}'\right)+\sum_{j\in J}a_{j}\left(\xi_{j}-\xi_{j}'\right).
\]
Thus, when conditioning on any fixed outcome of $\xi_{J},\xi_{J'}$, we can interpret $f(\xi_{I},\xi_{J})-f(\xi_{I},\xi_{J}')$ as a linear function in $\xi_I$ with coefficients $A_\ell$. Hence, for any $t\in \RR$, we obtain
\begin{equation}\label{eq:cos}
\left|\E[e^{2\pi it\left(f(\xi_{I},\xi_{J})-f(\xi_{I},\xi_{J}')\right)}\mid \xi_{J},\xi_{J'}]\right|=\prod_{\ell\in I}\left|\frac{e^{2\pi itA_{\ell}}+e^{-2\pi itA_{\ell}}}{2}\right|
 =\prod_{\ell\in I}\left|\cos\left(2\pi tA_{\ell}\right)\right|.
\end{equation}

Now, for real $\posreal \ge0$, $t\in [-1,1]\setminus \{0\}$ and $\ell\in I$, let $\mathcal{E}_{\ell}^{\posreal ,t}$ be the event that $|2A_{\ell}-\zint/t|\le \posreal $
for some integer multiple $\zint/t$ of $1/t$. Note that if $\mathcal{E}_{\ell}^{\posreal ,t}$
does not hold\footnote{Here we are implicitly assuming that $\posreal |t|< 1$, because if $\posreal |t|\ge 1$, then $\mathcal{E}_{\ell}^{\posreal ,t}$ always holds: if $\posreal |t|\ge 1$, we can always find an integer $\zint\in \ZZ$ with $|2A_{\ell}t-\zint|\le \posreal |t|$.} then $\left|\cos\left(2\pi tA_{\ell}\right)\right|=e^{-\Omega(t^{2}\posreal^{2})}$. We will now use a concentration inequality and Hal\'asz' inequality (\cref{thm:halasz}) to find an upper bound for the probability of the event that $\mathcal{E}_{\ell}^{\posreal ,t}$ holds for many different $\ell$.

\begin{claim}
\label{claim:Ekt-many-i}
Consider some $r$-tuple $(i_{1},\dots,i_{r})\in T_{I}\su I^r$,
and let $\mathcal{E}_{i_{1},\dots,i_{r}}^{\posreal ,t}=\mathcal{E}_{i_{1}}^{\posreal ,t}\cap\dots\cap\mathcal{E}_{i_{r}}^{\posreal ,t}$. Then $\Pr(\mathcal{E}_{i_{1},\dots,i_{r}}^{\posreal ,t})=O\left(p\left(\posreal ,t\right)\right)$,
where
$
p\left(\posreal ,t\right)=\left(\left|t\right|\log n+1/\sqrt n\right)^r\left(\posreal +1\right)^{r}
$.
\end{claim}
\begin{proof}
We may assume that $\posreal |t|\le 1$ as otherwise the claim is trivial.
Let $J^*$ be the subset of indices $j\in J$ such that $(\xi_{j}-\xi_{j}')\ne 0$. This is a uniformly random subset of $J$, so by the Chernoff bound and \cref{lem:inherit-non-degeneracy}, with probability $1-e^{-\Omega(n)}$, we have $|J^*|\ge |J|/3\ge n/9$ and $M_{J^*}(i_1,\dots,i_r)$ is $(\delta/9)$-non-degenerate. Condition on such an outcome of $J^*$.

Now, conditionally, the random variables $\xi_j^*:=\left(\xi_{j}-\xi_{j}'\right)/2$, for $j\in J^*$, are Rademacher distributed and mutually independent. For $j\in J^*$ let $b_j=(a_{i_1j},\dots,a_{i_rj})$ be the column of $M_{J^*}(i_1,\dots,i_r)$ indexed by $j\in J^*$. Recall that for $q=1,\dots,r$, we have $A_{i_q}=\sum_{j\in J}a_{i_qj}(\xi_j-\xi_j')=2\sum_{j\in J}a_{i_qj}\xi_j^*$. Hence the vector $(A_{i_1},\dots,A_{i_q})\in \RR^r$ equals $2\sum_{j\in J}\xi_j^*b_j$. Thus, by Hal\'asz' inequality (\cref{thm:halasz}), still conditioning on our outcome of $J^*$, for each $v=(\zint_{1},\dots,\zint_{r})\in\ZZ^{r}$ we have 
\[
\Pr\left(\left|2A_{i_{q}}-\zint_{q}/t\right|\le \posreal \text{ for all }q=1,\dots,r\right)=\Pr\left(\left\Vert\sum_{j\in J^*} \xi_j^* b_j-\frac{v}{4t}\right\Vert_\infty\le \frac \posreal 4\right)=O\left(\frac{\left(\posreal +1\right)^{r}}{n^{r/2}}\right).
\]
(Note that the above equation features the norm $\Vert\cdot\Vert_\infty$, while \cref{thm:halasz} concerns the norm $\Vert\cdot\Vert_2$. We can cover a $r$-dimensional box of side-length $\posreal /2$ with $O(\posreal +1)^r$ balls of radius $\delta/9$, and then we can apply \cref{thm:halasz} to each of these balls using that the matrix $M_{J^*}(i_1,\dots,i_r)$ is $(\delta/9)$-non-degenerate.) Also, by the Azuma--Hoeffding inequality and the union bound,
\[
\Pr\left(\left|2A_{i_{q}}\right|\ge \sqrt{n}\log n\text{ for some }q\in [r]\right)=e^{-\Omega((\log n)^2)}.
\]
Finally, note that there are $O\left(\left(\left|t\right|(\sqrt{n}\log n+\posreal )+1\right)^r\right)$ choices $(\zint_{1},\dots,\zint_{r})\in \ZZ^r$ such that each $|\zint_q/t|\le \sqrt{n}\log n+\posreal $, so we can conclude that
\[
\Pr\left(\mathcal{E}_{i_{1},\dots,i_{r}}^{\posreal ,t}\right)=O\left(\left(\left|t\right|(\sqrt{n}\log n+\posreal )+1\right)^r\frac{\left(\posreal +1\right)^{r}}{n^{r/2}}+e^{-\Omega((\log n)^2)}+e^{-\Omega(n)}\right)=O\left(p\left(\posreal ,t\right)\right).\tag*{\qedhere}
\]
\end{proof}
Now, let $W^{\posreal ,t}$ be the number of $r$-tuples $(i_{1},\dots,i_{r})\in T_{I}$
satisfying $\mathcal{E}_{i_{1},\dots,i_{r}}^{\posreal ,t}$. By \cref{claim:Ekt-many-i}, $\E W^{\posreal ,t}=O\left(p(\posreal ,t)|T_I|\right)$,
so $\Pr\left(W^{\posreal ,t}\ge |T_I|/2\right)=O\left(p\left(\posreal ,t\right)\right)$
by Markov's inequality. But observe that if $W^{\posreal ,t}<|T_I|/2$, then at least $|T_I|/2$ different $r$-tuples $(i_{1},\dots,i_{r})\in T_{I}$ contain some index $i_q$ such that $\mathcal{E}_{i_{q}}^{\posreal ,t}$ does not hold. As the $r$-tuples in $T_I$ are all disjoint, this means that there are at least $|T_I|/2=\Omega(n)$ indices $\ell\in I$ such that $\mathcal{E}_{\ell}^{\posreal ,t}$ does not hold, so $\prod_{\ell\in I}\left|\cos\left(2\pi tA_{\ell}\right)\right|=e^{-\Omega(t^{2}\posreal^{2}n)}$.

We have proved that for any $x\in (0,1)$, there is $\posreal =O\left(\sqrt{-\log(x)/(t^{2}n)}\right)$ such that
\begin{equation}
\Pr\left(\prod_{\ell\in I}\left|\cos\left(2\pi tA_{\ell}\right)\right|\geq x\right)\leq \Pr\left(W^{\posreal ,t}\ge |T_I|/2\right)=O\left(p\left(\posreal ,t\right)\right).\label{eq:to-pkt}
\end{equation}
For this value of $\posreal $, we compute
\begin{align}
p\left(\posreal ,t\right)&=O\left(\left(\left|t\right|\log n+1/{\sqrt n}\right) \left(\sqrt{\log(1/x)/(t^{2}n)}+1\right)\right)^{r}\notag\\
&=\begin{cases}
{\displaystyle O\left(\left(\frac{\log n}{\sqrt n}\right)^{r}+\left(\log\left(1/x\right)\right)^{r/2}\left(\frac{\log n}{|t|n}\right)^{r}\right)} & \text{for }\left|t\right|\le{\displaystyle \frac{1}{\sqrt{n}}},\vspace{5pt}\\
{\displaystyle O\left(\left(|t|\log n\right)^{r}+\left(\log\left(1/x\right)\right)^{r/2}\left(\frac{\log n}{\sqrt n}\right)^{r}\right)} & \text{for }\left|t\right|\ge{\displaystyle \frac{1}{\sqrt{n}}}.
\end{cases}\label{eq:pkt}
\end{align}
Note that
\begin{equation}
\int_{0}^{1}\left(\log\left(1/x\right)\right)^{r/2}\d x\le \sum_{j=1}^{\infty}\int_{2^{-j}}^{2^{-j+1}}\left(\log_2\left(1/x\right)\right)^{r/2}\d x\leq \sum_{j=1}^{\infty}2^{-j}j^{r/2}=O\left(1\right).\label{eq:integrate-log}
\end{equation}
Combining \cref{eq:decoupled,eq:cos,eq:to-pkt,eq:pkt,eq:integrate-log}, we have
\begin{align}
\left|\E e^{2\pi itf(\xi)}\right|^{2} &\le\E\left[\prod_{\ell\in I}\left|\cos\left(2\pi tA_{\ell}\right)\right|\right]\notag\\
&=\int_{0}^{1}\Pr\left(\prod_{\ell\in I}\left|\cos\left(tA_{\ell}\right)\right|\ge x\right)\d x\notag\\
 &=\begin{cases}
{\displaystyle O\left(\left(\frac{\log n}{\sqrt n}\right)^{r}+\left(\frac{\log n}{|t|n}\right)^{r}\right)=O\left(\vphantom{\frac{\frac{\int}{\int}}{\frac{\int}{\int}}}\left(\frac{\log n}{\left|t\right|n}\right)^{r}\right)} & \text{for }\left|t\right|\le{\displaystyle \frac{1}{\sqrt{n}}},\vspace{5pt}\\
{\displaystyle O\left(\left(|t|\log n\right)^{r}+\left(\frac{\log n}{\sqrt n}\right)^{r}\right)=O\left(\vphantom{\int}\left(\left|t\right|\log n\right)^{r}\right)} & \text{for }\left|t\right|\ge{\displaystyle \frac{1}{\sqrt{n}}}.
\end{cases}\label{eq:characteristic-function-bound}
\end{align}
Finally, we apply Ess\'een's concentration inequality (\cref{lem:esseen}) with $s=n^{2/(r+2)}$ and $\eps=1/s$ (note that we need to take the square root of the bound in \cref{eq:characteristic-function-bound}). This yields
\begin{align*}
\Pr\left(\left|f(\xi)-x\right|\le n^{2/(r+2)}\right) &\le O\left(s\right)\int_{-1/s}^{1/s}\vert\E e^{2\pi itf(\xi)}\vert\d t\\
&=O\left(\frac{s}{n}+s\int_{1/n}^{1/\sqrt{n}}\left(\frac{\log n}{t n}\right)^{r/2}\d t+s\int_{1/\sqrt{n}}^{1/s}\left(\log n\right)^{r/2}t^{r/2}\d t\right)\\
&=O\left(\frac{s}{n}+s\left(\frac{\log n}{n}\right)^{r/2}\cdot (1/n)^{-r/2+1}+s(\log n)^{r/2}\cdot (1/s)^{r/2+1}\right)\\
&=O\left((\log n)^{r/2}\left(\frac{s}{n}+s^{-r/2}\right)\right)=O\left(\left(\log n\right)^{r/2}n^{-r/(r+2)}\right).\tag*{\qedhere}
\end{align*}
\end{proof}

\section{Real projections of complex nonsingular matrices}\label{sec:complex-to-real}

The anti-concentration inequality in the last section (\cref{lem:key}) was only for quadratic polynomials with real coefficients, whereas in \cref{thm:rank-L1,thm:rank-edit} we wish to consider complex polynomials as well. The following lemma will be an important tool to reduce from the complex case to the real case, and may be of independent interest. Recall that matrix norms in this paper are entrywise.
\begin{lem}
\label{lem:L1-real-to-complex}For every integer $r\geq 1$ and any $\eps>0$, there exists $c=c(r,\eps)>0$ such that the following holds. Let $A$ be a complex $r\times r$ matrix with $\vert \det A\vert\geq \eps$ and $\Vert A\Vert_\infty\le 1$. Let $\theta\in\left[-\pi,\pi\right]$ be a uniformly random phase. Then with probability at least $c$, the matrix $\Re\left(e^{i\theta}A\right)$ satisfies $\vert \det \Re\left(e^{i\theta}A\right)\vert\geq c$.
\end{lem}

\begin{proof}
Define the polynomial $p(z)=\det\left(z^{2}A+\overline{A}\right)$,
then
\[
\det \Re(e^{i\theta}A)=\det\left(\frac{e^{i\theta}A+e^{-i\theta}\overline{A}}{2}\right)=\frac{e^{-ir\theta}}{2^r}p(e^{i\theta}).
\]
Note that $\vert p(0)\vert=\vert \det\overline{A}\vert=\vert \det A\vert\geq \eps$, and let $\phi\in [-\pi,\pi]$ be the phase of $p\left(0\right)$. Now, for random $\theta\in\left[-\pi,\pi\right]$, consider the expression $e^{-i\phi}p(e^{i\theta})$. This expression can be written as a linear combination of powers $e^{ik\theta}$ for $k\geq 0$ with  complex coefficients. Note that  we have $\E\Re\left(a\cdot e^{ik\theta}\right)=0$ for each positive integer $k$ and each $a\in \CC$. Thus, $\E\Re\left(e^{-i\phi}p(e^{i\theta})\right)$ is purely determined by the constant coefficient of the polynomial $p$ and we have $\E\Re\left(e^{-i\phi}p(e^{i\theta})\right)=\Re\left(e^{-i\phi}p\left(0\right)\right)$. Thus,
\[
\E\left|\det \Re(e^{i\theta}A)\right|=2^{-r}\cdot \E\left|p(e^{i\theta})\right|\ge2^{-r}\cdot \E\Re(e^{-i\phi}p(e^{i\theta}))=2^{-r}\cdot \Re\left(e^{-i\phi}p\left(0\right)\right)=2^{-r}\cdot \vert p(0)\vert\ge\varepsilon2^{-r}.
\]
On the other hand, since the entries of the matrix $\Re\left(e^{i\theta}A\right)$
each have absolute value at most 1, we have $\left|\det\Re\left(e^{i\theta}A\right)\right|\le r!$.
So, by Markov's inequality,
\begin{align*}
\Pr\left(\left|\det \Re(e^{i\theta}A)\right|\ge c\right) & =1-\Pr\left(r!-\left|\det\Re(e^{i\theta}A)\right|> r!-c\right)\ge1-\frac{r!-\varepsilon2^{-r}}{r!-c}.
\end{align*}
For sufficiently small $c>0$, this probability is at least $c$.
\end{proof}

In this section we also prove that if a low-dimensional matrix has bounded entries and determinant bounded away from zero, then its least singular value is bounded away from zero as well.

\begin{lem}\label{lem-Be-large}
For some integer $q\geq 1$, let $B$ be a complex nonsingular $q\times q$ matrix with $\Vert B\Vert_\infty\le 1$. Then for any unit vector $e\in \CC^q$, the vector $Be$ satisfies $\Vert Be\Vert_1\geq (q!)^{-1}\cdot \vert \det B\vert$.
\end{lem}
\begin{proof}
Let $v=Be$. First, we have $\Vert B^{-1}v\Vert_1=\Vert e\Vert_1\ge \Vert e\Vert_2=1$. On the other hand, observe that $B^{-1}$ can be calculated from the determinant of $B$ and the adjugate matrix of $B$. All entries of the adjugate matrix of $B$ have absolute value at most $(q-1)!$, and therefore all entries of $B^{-1}$ all have absolute value at most $(q-1)!\cdot \vert \det B\vert^{-1}$. Thus, each entry of the vector $e=B^{-1}v$ has absolute value at most $(q-1)!\cdot\vert \det B\vert^{-1}\cdot\Vert v\Vert_1$. It follows that
\[1\le \Vert B^{-1}v\Vert_1\leq q\cdot (q-1)!\cdot\vert \det B\vert^{-1}\cdot \Vert v\Vert_1,\]
from which the desired result immediately follows.
\end{proof}

\section{Deducing the main theorems\label{sec:rank}}

In this section we explain how to prove \cref{thm:rank-L1,thm:rank-edit}. Before getting into the details, we first observe that \cref{thm:rank-edit} actually follows from a slight variant of \cref{thm:rank-L1}, where we control ``coefficient-$L_1$-norm'' but we demand that certain coefficients lie in a certain finite set.

\begin{thm}
\label{thm:rank-L1-set}Let $\FF\in \lbrace \CC,\RR, \QQ\rbrace$. For any integer $r\geq 3$, any $0< \varepsilon\leq 1$, and any finite set $S\su \FF$ with $\vert s\vert\leq 1$ for all $s\in S$, there is a constant $C=C(r,\varepsilon, S)$ and a finite set $S^*=S^*(r, S)\su \FF$ such that the following holds. Let $f\in \FF[x_1,\dots,x_n]$ be a quadratic polynomial, let $\tilde{f}$ be the homogeneous degree-2 part of $f$ and assume that the coefficients of $\tilde{f}$ are elements of the set $S$. Let $\xi=(\xi_1,\dots,\xi_n)\in \Rad^n$, and suppose that we have
\[\sup_{x\in \FF}\Pr\left(f(\xi)=x\right)\geq C\cdot \frac{\left(\log n\right)^{r/2}}{n^{1-2/\left(r+2\right)}}.\]
Then there is a quadratic form $h\in\FF\left[x_{1},\dots,x_{n}\right]$
of rank strictly less than $r$ such that the sum of the absolute values of the coefficients of $\tilde{f}-h$ is at most $\varepsilon n^2$, and such that all coefficients of $h$ are elements of the set $S^*$.
\end{thm}

\begin{proof}[Proof of \cref{thm:rank-edit} given \cref{thm:rank-L1-set}]
First of all, by rescaling we may assume that the set $S$ in \cref{thm:rank-edit} satisfies $\vert s\vert\leq 1$ for all $s\in S$. Now, let $S^*=S^*(r,S)$ be the finite set in \cref{thm:rank-L1-set}, and let $\Delta$ be the minimum distance between two elements of $S^*\cup S$.

Let $f\in \FF[x_1,\dots,x_n]$ be as in \cref{thm:rank-edit}, and note that we may assume that $n$ is large with respect to $\eps$. Let $\tilde{f}$ be the homogeneous degree-2 part of $f$. By \cref{thm:rank-L1-set} applied with the error parameter $\varepsilon\cdot \Delta/2$, we either have the desired inequality on point probabilities of $f(\xi)$, or there is a quadratic form $h\in\FF\left[x_{1},\dots,x_{n}\right]$ of rank less than $r$ with coefficients in $S^*$ such that the sum of absolute values of $\tilde{f}-h$ is at most $\varepsilon\cdot (\Delta/2)\cdot n^2$. By the choice of $\Delta$, this implies that $h$ and $\tilde{f}$ differ in at most $(\varepsilon/2)\cdot n^2$ coefficients. Thus, $h$ and $f$ differ in at most $(\varepsilon/2)\cdot n^2+n+1\leq \varepsilon n^2$ coefficients (if $n$ is sufficiently large).
\end{proof}

Now, for the proofs of \cref{thm:rank-L1,thm:rank-L1-set} we will need a robust version of linear independence.

\begin{defn}\label{def:eps-independence}
Consider $\FF\in \lbrace \CC,\RR, \QQ\rbrace$. For any $0\leq \eps\leq 1$, let us say that vectors $v_1,\dots,v_q\in \FF^n$ are $\eps$-dependent (over $\FF$) if there are linearly dependent vectors $v_1',\dots,v_q'\in \FF^n$ with
\begin{equation*}
\Vert v_1-v_1'\Vert_1+\dots +\Vert v_q-v_q'\Vert_1\leq \eps n.
\end{equation*}
Otherwise, say that $v_1,\dots,v_q$ are $\eps$-independent.
\end{defn}

Note that the usual notion of being linearly independent corresponds to being 0-independent. Also note that the empty collection of vectors (taking $q=0$) is $\eps$-independent for any $0\leq \eps\leq 1$. In \cref{sec:aux} we will observe some more basic properties of $\eps$-independence, and prove some analogues of simple facts about ordinary linear independence.

Recall that we already introduced the notion of being $\delta$-non-degenerate in \cref{def:non-degenerate}, which can also be interpreted as a type of robust linear independence. The following lemma connects these two notions.

\begin{lem}\label{lem-independent-non-degenerate-real}
For any integer $r\geq 1$ and any $0<\eps\leq 1$, there is a constant $c=c(r,\eps)>0$ such that the following holds. Suppose that $v_1,\dots,v_r\in \RR^n$ are $\eps$-independent vectors with $\Vert v_i\Vert_\infty\le 1$ for each $i$. Then the $r\times n$ matrix with rows $v_1,\dots,v_r$ is $c$-non-degenerate.
\end{lem}

We will also need a variant of \cref{lem-independent-non-degenerate-real} for complex vectors, incorporating our lemma concerning real projections of complex matrices (\cref{lem:L1-real-to-complex}).

\begin{lem}\label{lem-independent-non-degenerate-complex}
For any integer $r\geq 1$ and any $0<\eps\leq 1$, there is a constant $c=c(r,\eps)>0$ such that the following holds. Suppose $v_1,\dots,v_r\in \CC^n$ are $\eps$-independent vectors with $\Vert v_i\Vert_\infty\le 1$ for each $i$. Furthermore, let $\theta\in [-\pi, \pi]$ be a uniformly random phase. Then with probability at least $c$, the $r\times n$ matrix with rows $\Re(e^{i\theta}v_1),\dots,\Re(e^{i\theta}v_r)\in \RR^n$ is $c$-non-degenerate.
\end{lem}

We defer the proofs of \cref{lem-independent-non-degenerate-real,lem-independent-non-degenerate-complex} to \cref{sec:aux}.

Now, the remaining ingredient for the proofs of \cref{thm:rank-L1,thm:rank-L1-set} is the following lemma. It states that if a symmetric matrix does not ``robustly'' have rank at least $r$, then it must be close (in terms of entrywise $L_1$-norm) to a symmetric matrix of rank less than $r$.

\begin{lem}\label{lem-matrix-rank-close}
Fix $\FF\in \lbrace \CC,\RR, \QQ\rbrace$ and an integer $r\geq 1$. Let $0<\alpha\leq 1$ and $0<\delta<1/r$ and let $A\in \FF^{n\times n}$ be a symmetric matrix with $\Vert A\Vert_\infty\le 1$. Suppose that there do not exist $\delta n$ disjoint $r$-tuples of $\alpha$-independent rows of $A$. Then there exists a symmetric matrix $H\in \FF^{n\times n}$ of rank less than $r$ such that $\Vert A-H\Vert_1\leq O(\delta+\alpha^{(6r)^{-r}})\cdot  n^2$.
\end{lem}

Here, the implicit constant in the $O$-term may depend on $r$. We defer the proof of \cref{lem-matrix-rank-close} to \cref{sec:lem-matrix-rank-close}.

We remark that \cref{lem-matrix-rank-close} also implies the following corollary, which may be of independent interest: if a symmetric matrix is close to a matrix which has rank less than $r$, then it is close to a symmetric matrix with rank less than $r$.

\begin{cor}
\label{cor:symmetric-close}
Fix $\FF\in \lbrace \CC,\RR, \QQ\rbrace$ and an integer $r\geq 1$. Let $0<\alpha\leq 1$ and let $A\in \FF^{n\times n}$ be a symmetric matrix with $\Vert A\Vert_\infty\le 1$. Suppose that there is a matrix $A'\in \FF^{n\times n}$ of rank less than $r$ with $\Vert A-A'\Vert_1\leq \alpha n^2$. Then there exists a symmetric matrix $H\in \FF^{n\times n}$ of rank less than $r$ such that $\Vert A-H\Vert_1\leq O(\alpha^{1/(2\cdot (6r)^{r})})\cdot n^2$.
\end{cor}

\begin{proof}
It suffices to prove that $A$ cannot have $\alpha^{1/2} n$ disjoint $r$-tuples of $\alpha^{1/2}$-independent rows; we would then be able to apply \cref{lem-matrix-rank-close}. So, suppose there were $\alpha^{1/2} n$ disjoint $r$-tuples $(i_1,\dots,i_r)\in [n]^r$ such that the rows $\row_{i_1}(A),\dots, \row_{i_r}(A)$ are $\alpha^{1/2}$-independent.

As in \cref{def:non-degenerate}, for each of our $r$-tuples $(i_1,\dots,i_r)$ we use notation like $A(i_1,\dots,i_r)$ to represent the $r\times n$ submatrix of $A$ with these rows. Since our $r$-tuples are disjoint, for at least one of our $r$-tuples $(i_1,\dots,i_r)$ we have $\Vert A(i_1,\dots,i_r)-A'(i_1,\dots,i_r)\Vert_1\le \Vert A-A'\Vert_1/(\alpha^{1/2} n)\leq \alpha^{1/2} n^2$. Since $A'$ has rank less than $r$, the rows of $A'(i_1,\dots,i_r)$ are linearly dependent, so the rows of $A(i_1,\dots,i_r)$ are $\alpha^{1/2}$-dependent, a contradiction.
\end{proof}

In order to prove \cref{thm:rank-L1-set}, we also need the following variant of \cref{lem-matrix-rank-close}, stating that if the entries of $A$ lie in a finite set $S$, then the matrix $H$ can be chosen such that its entries lie in a finite set $S'$ (depending only on $S$ and $r$).

\begin{lem}\label{lem-matrix-rank-close-set}
For $\FF\in \lbrace \CC,\RR, \QQ\rbrace$, an integer $r\geq 1$, and a finite set $S\su \FF$ with $\vert s\vert\leq 1$ for all $s\in S$, there is a finite set $S'=S'(r,S)\su \FF$ such that the following holds: let $0<\alpha\leq 1$ and $0<\delta<1/r$ and let $A\in \FF^{n\times n}$ be a symmetric matrix all of whose entries are in $S$. Suppose that there do not exist $\delta n$ disjoint $r$-tuples of $\alpha$-independent rows of $A$. Then there exists a symmetric matrix $H\in \FF^{n\times n}$ of rank less than $r$ such that $\Vert A-H\Vert_1\leq O(\delta+\alpha^{(6r)^{-r}})\cdot  n^2$ and such that all entries of $H$ are elements of $S'$.
\end{lem}

Now, we can deduce \cref{thm:rank-L1,thm:rank-L1-set}. The proofs are virtually the same (the only difference is whether we use \cref{lem-matrix-rank-close-set} or \cref{lem-matrix-rank-close}), so we present both proofs together.

\begin{proof}
[Proof of \cref{thm:rank-L1,thm:rank-L1-set}]
First, choose some small $0<\alpha\leq 1$ and $0<\delta<1/r$ such that the $O(\delta+\alpha^{(6r)^{-r}})$-term in \cref{lem-matrix-rank-close,lem-matrix-rank-close-set} is at most $\eps$. Also, note that we may assume $n$ is sufficiently large with respect to $\eps$. For each $i,j$, let $a_{ij}$ be the coefficient of $x_ix_j$ and in $f$ (so $a_{ii}$ is the coefficient of $x_i^2$), and define the (symmetric) ``coefficient matrix'' $A=(a_{ij})_{i,j}$. Note that by our assumptions on $f$, $\Vert A\Vert_\infty\le 1$. We consider two cases.

{\bf Case 1:} $A$ does not have $\delta n$ disjoint $r$-tuples of $\alpha$-independent rows

By \cref{lem-matrix-rank-close} (and our choice of $\delta$ and $\alpha$), there exists a symmetric matrix $H=(h_{ij})_{1\leq i,j\leq n}\in \FF^{n\times n}$ of rank less than $r$ such that $\Vert A-H\Vert_1\leq\eps n^2$. In the setting of \cref{thm:rank-L1-set}, we can instead apply \cref{lem-matrix-rank-close-set} to get the same conclusion, with the additional property that all entries of $H$ lie in some fixed finite set $S'$ that only depends on $r$ and $S$.

Now, let $h\in \FF[x_1,\dots,x_n]$ be the quadratic form defined by $h(x)=\frac{1}{2}x^T H x=\sum_{i<j} h_{ij}x_ix_j+\sum_i (h_{ii}/2)x_i^2$, which also has rank less than $r$ (and in the setting of \cref{thm:rank-L1-set}, the coefficients of $h$ lie in the finite set $S^*=S'\cup \lbrace s/2: s\in S'\rbrace$). Let $\tilde f$ be the homogeneous degree-2 part of $f$. Then, the sum of the absolute values of the coefficients of $\tilde f-h$ is
\[
\sum_{i<j} |a_{ij}-h_{ij}|+\sum_i  |a_{ii}-h_{ii}/2|\le \frac{1}{2}\Vert A-H\Vert_1+\sum_i |a_{ii}/2|\le (\eps/2)n^2+n/2.
\]

For the proof of \cref{thm:rank-L1-set}, this already gives the desired conclusion (for sufficiently large $n$). In the setting of \cref{thm:rank-L1}, we additionally note that the sum of the absolute values of the coefficients of $\tilde f-f$ is at most $n+1=o(n^2)$.

{\bf Case 2:} $A$ has $\delta n$ disjoint $r$-tuples of $\alpha$-independent rows\nopagebreak

In this case, we use \cref{lem:key} to prove that
\[\sup_{x\in \FF}\Pr\left(f(\xi)=x\right)=O\left(\frac{\left(\log n\right)^{r/2}}{n^{1-2/\left(r+2\right)}}\right),\]
showing that the assumptions of \cref{thm:rank-L1,thm:rank-L1-set} cannot hold for large $C$. (For the rest of the proof, all asymptotic notation treats $r$, $\delta$ and $\alpha$ as fixed constants.) Let $T_A\su [n]^r$ be a collection of $\delta n$ disjoint $r$-tuples, such that for any $(i_1,\dots,i_r)\in T_A$ the corresponding rows of $A$ are $\alpha$-independent.

If $\FF=\RR$ or $\FF=\QQ$, then by \cref{lem-independent-non-degenerate-real}, for each $(i_1,\dots,i_r)\in T_A$, the $r\times n$ matrix of $A(i_1,\dots,i_r)$, defined as in \cref{def:non-degenerate}, is $c$-non-degenerate for $c=\Omega(1)$. Thus, \cref{lem:key} implies that for random $\xi=(\xi_1,\dots,\xi_n)\in \Rad^n$, we have
\[\sup_{x\in \FF}\Pr\left(f(\xi)=x\right)\leq \sup_{x\in \FF}\Pr\left(\left|f\left(\xi\right)-x\right|\le n^{2/\left(r+2\right)}\right)=O\left(\frac{\left(\log n\right)^{r/2}}{n^{1-2/\left(r+2\right)}}\right),\]
as desired.

For the case $\FF=\CC$, choose $c=\Omega(1)$ such that if $\theta\in [-\pi,\pi]$ is a uniformly random phase then for each $(i_1,\dots,i_r)\in T_A$, the matrix $\Re(e^{i\theta}A)(i_1,\dots,i_r)$ is $c$-degenerate with probability at least $c$. Such a $c$ exists by \cref{lem-independent-non-degenerate-complex}. Then, let $T_\theta$ be the set of all $(i_1,\dots,i_r)\in T_A$ such that $\Re(e^{i\theta}A)(i_1,\dots,i_r)$ is $c$-degenerate. For random $\theta$ we have $\E|T_\theta|\ge c\delta n$, so we can fix $\theta$ such that $|T_\theta|\ge c\delta n$. Then, let $f^*$ be the polynomial obtained from $e^{i\theta} f$ by taking the real part of each coefficient, so for $\xi\in \Rad^n$ \cref{lem:key} implies
\[\sup_{x\in \CC}\Pr\left(f(\xi)=x\right)\leq\sup_{x\in \CC}\Pr\left(\Re(e^{i\theta}f(\xi))=\Re(e^{i\theta}x)\right)= \sup_{x\in \RR}\Pr\left(f^*(\xi)=x\right)=O\left(\frac{\left(\log n\right)^{r/2}}{n^{1-2/\left(r+2\right)}}\right),\]
as desired.
\end{proof}

\section{Lemmas on robust linear independence\label{sec:aux}}

In this section we prove \cref{lem-independent-non-degenerate-real,lem-independent-non-degenerate-complex}, and several other auxiliary lemmas concerning $\varepsilon$-independence (defined in \cref{def:eps-independence}). Throughout this section we fix $\FF\in \lbrace \CC, \RR, \QQ\rbrace$.

First, with a view towards proving \cref{lem-independent-non-degenerate-real,lem-independent-non-degenerate-complex}, we start with the fact that if the rows of a $q\times n$ matrix are $\varepsilon$-independent then there is a $q\times q$ submatrix with large determinant.

\begin{lem}\label{lem-large-det-minor}
For any $0\leq \eps\leq 1$ and any $\eps$-independent vectors $v_1,\dots,v_q\in \FF^n$, the $q\times n$ matrix with rows $v_1,\dots,v_q$ has a $q\times q$ submatrix whose determinant has absolute value at least $\eps^q$.
\end{lem}

\cref{lem-large-det-minor} will be an immediate consequence of a more general lemma (\cref{lem-large-det-minor-graph-version}) that we prove later in this section. Now we prove \cref{lem-independent-non-degenerate-real}.

\begin{proof}[Proof of \cref{lem-independent-non-degenerate-real}]
We will take $c=\eps^r/(2^r r!r)$. Let $M$ be the $r\times n$ matrix with rows $v_1,\dots,v_r$.

First, we claim that $M$ has $m\geq \eps/(2r^2)\cdot n$ disjoint $r\times r$ submatrices $B_1,\dots,B_m$ whose determinants have absolute value at least $(\eps/2)^r$. Indeed, consider a maximal collection such disjoint submatrices and suppose that this collection consists of fewer than $\eps/(2r^2)\cdot n$ submatrices. Let $M'\in \RR^{r\times n}$ be obtained from $M$ by setting all the entries in our identified submatrices to zero. By maximality, every $r\times r$ submatrix of $M'$ has determinant bounded in absolute value by $(\eps/2)^r$. On the other hand, since $\Vert M\Vert_\infty\le 1$, we have $\Vert M-M'\Vert_1\leq \eps/(2r^2)\cdot n\cdot r^2\leq (\eps/2) n$. Since the rows of $M$ are $\eps$-independent, this shows that the rows of $M'$ are $(\eps/2)$-independent. But then, by \cref{lem-large-det-minor}, the matrix $M'$ has a $r\times r$ submatrix whose determinant has absolute value at least $(\eps/2)^r$, which is a contradiction.

Now, in order to show that $M$ is $c$-non-degenerate, we need to show that for every unit vector $e\in \RR^r$, there are at least $c n$ columns $w$ of $M$ such that $\vert \langle w,e\rangle \vert\geq c$. As $m\geq \eps/(2r^2)\cdot n\geq cn$, it suffices to show that for each $j=1,\dots,m$ there is a column $w$ of $B_j$ with $\vert\langle w,e\rangle\vert\geq c$. This is equivalent to showing that the vector $B_j^{T} e\in \RR^r$ has at least one entry with absolute value at least $c$. However, since $\vert \det B_j^T\vert\geq (\eps/2)^r$, \cref{lem-Be-large} implies that $\Vert B_j^{T} e\Vert_1\geq (r!)^{-1}\cdot (\eps/2)^r= r\cdot c$. Therefore one of the $r$ entries of $B_j^{T} e$ must indeed have absolute value at least $c$. This finishes the proof of \cref{lem-independent-non-degenerate-real}.
\end{proof}

Next, to prove \cref{lem-independent-non-degenerate-complex}, we modify the proof \cref{lem-independent-non-degenerate-real} to incorporate our lemma concerning real projections of complex matrices (\cref{lem:L1-real-to-complex}).

\begin{proof}[Proof of \cref{lem-independent-non-degenerate-complex}]
Let $M$ be the $r\times n$ matrix with rows $v_1,\dots,v_r$. As in the proof of \cref{lem-independent-non-degenerate-real}, in $M$ we can find $m\geq \eps/(2r^2)\cdot n$ disjoint $r\times r$ submatrices $B_1,\dots, B_m$ whose determinants have absolute value at least $(\eps/2)^r$.

Consider a random phase $\theta\in [-\pi, \pi]$. By \cref{lem:L1-real-to-complex}, for some $0<c'<1$ only depending on $r$ and $\eps$, for each $1\le j\le m$, with probability at least $c'$ we have $|\det \Re(e^{i\theta}B_j)|\ge c'$. Let $J$ be the random set of $j$ such that this holds, and observe that $\E |J|\ge c'm$. Then, since $|J|-c'm/2\le m$ we have $m\Pr(|J|-c'm/2\ge 0)\ge \E[|J|-c'm/2]$, so
\[
\Pr(|J|\ge c'm/2)=\Pr(|J|-c'm/2\geq 0)\ge \frac{\E[|J|-c'm/2]}{m}\geq \frac{c'm-c'm/2}{m}=\frac{c'}{2}.
\]
But, if $|J|\ge c'm/2$ then $\Re(e^{i\theta}M)$ has $c'm/2\geq c'\eps/(4r^2)\cdot n$ disjoint $r\times r$ submatrices whose determinants have absolute value at least $c'$. As in the proof of \cref{lem-independent-non-degenerate-real}, it follows from \cref{lem-Be-large} that $\Re(e^{i\theta}M)$ is $c$-non-degenerate for some $0<c<c'/2$ depending only on $\eps$ and $r$ (via $c'$). The desired result follows.
\end{proof}

In the remainder of this section, we prove some simple facts about $\eps$-independence that will be useful for the proof of \cref{lem-matrix-rank-close} in \cref{sec:lem-matrix-rank-close}. From now on, fix any $\FF\in \lbrace \CC, \RR, \QQ\rbrace$.

The following lemma says that for $\eps$-dependent vectors $v_1,\dots,v_q$ with entries of absolute value at most 1, the vectors $v_1',\dots,v_q'$ in \cref{def:eps-independence} can be chosen in such a way that their entries have absolute value at most $q+1$.

\begin{lem}\label{lem-eps-dependent-change-bounded}
Let $0\leq \eps\leq 1$ and let $v_1,\dots,v_q\in \FF^n$ be $\eps$-dependent such that $\Vert v_i\Vert_\infty\le 1$ for each $i$. Then there are linearly dependent vectors $v_1',\dots,v_q'\in \FF^n$ with
\[\Vert v_1-v_1'\Vert_1+\dots +\Vert v_q-v_q'\Vert_1\leq \eps n\]
and such that all entries of the vectors $v_1',\dots,v_q'$ have absolute value at most $q+1$.
\end{lem}
\begin{proof}
By the definition of $v_1,\dots,v_q\in \FF^n$ being $\eps$-dependent, there are linearly dependent vectors $v_1',\dots,v_q'\in \FF^n$ with
\[\Vert v_1-v_1'\Vert_1+\dots +\Vert v_q-v_q'\Vert_1\leq \eps n.\]
It may be the case that for one or more indices $i$, the $i$-th entry of one of the vectors $v_1',\dots,v_q'$ has absolute value larger than $q+1$. For all such $i$, let us change the $i$-th entry of each of the vectors $v_1',\dots,v_q'$ to zero. It is not hard to see that this does not increase the value of $\Vert v_1-v_1'\Vert_1+\dots +\Vert v_q-v_q'\Vert_1$, and that the new vectors $v_1',\dots,v_q'$ are still linearly dependent.
\end{proof}

The next lemma is a generalisation of \cref{lem-large-det-minor} where we can specify some forbidden pairs of vectors.

\begin{lem}\label{lem-large-det-minor-graph-version}
Let $0\leq \eps\leq 1$ and $0\leq \delta\leq 1$. For some integer $q\geq 0$, let $M$ be a $q\times n$ matrix whose rows are $(\varepsilon+q(q-1)\delta)$-independent and whose entries have absolute value at most $1$. Consider a graph $G$ on the vertex set $[n]$ in which every vertex has degree at most $\delta n$. Then there exists a $q$-vertex independent set $I\su [n]$ of $G$, such that the $q\times q$ matrix $M_I$ satisfies $|\det M_I|\ge \eps^q$.
\end{lem}

Note that \cref{lem-large-det-minor} follows from \cref{lem-large-det-minor-graph-version} by taking $\delta=0$ and the graph $G$ with no edges.

\begin{proof}[Proof of \cref{lem-large-det-minor-graph-version}]
Let $\varepsilon'=\varepsilon+q(q-1)\delta$. We prove the lemma by induction on $q$. Note that the base case $q=0$ is trivial (since we adopted the convention that the determinant of the $0\times 0$ empty matrix is 1). So assume that $q\geq 1$ and that we have already proved the lemma for $q-1$. Let $M=(a_{ij})_{i,j}$ be the $q\times n$ matrix with rows $v_1,\dots,v_q$. As $v_1,\dots,v_q$ are $\eps'$-independent, the vectors $v_1,\dots,v_{q-1}$ are $\eps'$-independent as well and in particular $(\varepsilon+(q-1)(q-2)\delta)$-independent. Let $M'$ be the matrix obtained from $M$ by removing the last row. Then, by the induction hypothesis, there is a $(q-1)$-vertex independent set $I'\subseteq [n]$ of $G$ such that $|\det M_{I'}'|\ge \eps^{q-1}$. Let us assume without loss of generality that $I'=[q-1]$.

Now, let $J\su \lbrace q,\dots,n\rbrace$ be the set of those vertices which have a neighbour in $I'$ in the graph $G$. By the degree assumption, we have $|J|\le (q-1)\delta n$. Then, consider all the $q\times q$ submatrices of $M$ formed by the first $q-1$ columns together with a column which has index in $\{q,\dots,n\}\setminus J$. Assume for contradiction that for all of these submatrices the absolute value of their determinant is smaller than $\eps^q$.

We want to modify $M$ to create a matrix $M'$ contradicting the definition of $\eps'$-independence. First, for $j\in \{q,\dots,n\}\setminus J$, let us change the entry $a_{qj}$ in such a way that the determinant of the $q\times q$ submatrix of $M$ formed by the first $q-1$ columns together with the $j$-th column becomes zero. By the assumption in the last paragraph, we need to change $a_{qj}$ by at most $\eps^q/|\det M_{I'}'|\leq \eps^q/\eps^{q-1}=\eps$ in order to achieve this. Second, for every $j\in J$ change all the entries in column $j$ to zero. The resulting matrix $M'$ then satisfies $\Vert M-M'\Vert_1\le \eps (n-(q-1))+q(q-1)\delta n\leq \eps' n$. But by construction, the first $q-1$ columns of $M'$ span its entire column space, so $M'$ has rank less than $q$ and its rows are not linearly independent. This is in contradiction to the rows of $M$ being $\eps'$-independent.
\end{proof}

Next, the following lemma states that for robustly independent vectors $v_1,\dots,v_q\in \FF^n$ and a vector $v\in \FF^n$, either $v_1,\dots, v_q$ and $v$ are robustly independent together, or $v$ is close to a linear combination of $v_1,\dots,v_q$.

\begin{lem}\label{lem-indep-or-close-lin-comb}
Fix an integer $q\geq 0$. Let $0<\eps\leq 1$ and let $v_1,\dots,v_q\in \FF^n$ be $\eps$-independent vectors such that $\Vert v_i\Vert_\infty\le 1$ for each $i$. Furthermore let $0\leq\delta<\eps/2$ and let $v\in \FF^n$ be a vector with $\Vert v\Vert_\infty\le 1$. Then at least one of the following two conditions holds:
\begin{compactitem}
\item[(a)] $v_1,\dots, v_q,v$ are $\delta$-independent, or
\item[(b)] there is a vector $v^*\in \spn(v_1,\dots,v_q)$ with $\Vert v-v^*\Vert_1\leq O(\eps^{-q}\delta)\cdot n$.
\end{compactitem}
\end{lem}
Here the implicit constant in the $O$-term may depend on $q$.
\begin{proof}
Suppose (a) is not satisfied, so $v_1,\dots, v_q,v$ are $\delta$-dependent. Thus, by \cref{lem-eps-dependent-change-bounded} there are linearly dependent vectors $v_1',\dots,v_q',v'\in \FF^n$ with
\begin{equation}\Vert v_1-v_1'\Vert_1+\dots +\Vert v_q-v_q'\Vert_1+\Vert v-v'\Vert_1\leq \delta n,\label{eq:v'}\end{equation}
such that all entries of $v_1',\dots,v_q', v'$ have absolute value at most $q+1$. Now, since $v_1,\dots,v_q\in \FF^n$ are $\eps$-independent, and $\delta<\eps/2$, the vectors $v_1',\dots,v_q'$ are $(\eps/2)$-independent, and in particular linearly independent. Thus, as $v_1',\dots,v_q', v'$ are linearly dependent, we can write $v'$ as $v'=a_1v_1'+\dots+a_qv_q'$ for some $a_1,\dots,a_q\in \FF$.

We claim that $\vert a_i\vert=O(\eps^{-q})$ for all $i$. Indeed, since $v_1',\dots,v_q'$ are $(\eps/2)$-independent, by \cref{lem-large-det-minor} the matrix with rows $v_1',\dots,v_q'$ has a $q\times q$ submatrix $A$ whose determinant has absolute value $\Omega(\eps^q)$. Let $I\su [n]$ be the set of the indices of the columns contained in this submatrix. Consider the row vector $(v')_I\in \FF^q$ obtained from $v'$ by taking the coordinates indexed by $I$. Now, since $v'=a_1v_1'+\dots+a_qv_q'$ we have $(a_1,\dots,a_q)A=(v')_I$, so $(a_1,\dots,a_q)=(v')_IA^{-1}$. Since $\Vert A\Vert_\infty\le q+1=O(1)$, and $\vert\det A\vert =\Omega(\eps^q)$, we can see (from the formula for $A^{-1}$ in terms of the adjugate of $A$) that all entries of the matrix $A^{-1}$ are of the form $O(\eps^{-q})$. Since $\Vert (v')_I\Vert_\infty\le q+1=O(1)$ as well, we can conclude that the absolute values of $a_1,\dots,a_q$ are of the form $O(\eps^{-q})$, as claimed.

Now, define the linear combination $v^*=a_1v_1+\dots+a_qv_q$. We have
\[
\Vert v-v^*\Vert_1\leq\Vert v-v'\Vert_1+\Vert v'-v^*\Vert_1=\Vert v-v'\Vert_1+\left\Vert a_1(v_1'-v_1)+\dots+a_q(v_q'-v_q)\right\Vert_1.\]
Recalling \cref{eq:v'} and that $\vert a_i\vert=O(\eps^{-q})$ for all $i$, we deduce that $\Vert v-v^*\Vert_1\le O(\eps^{-q})\cdot\delta n,
$
so (b) holds.
\end{proof}

We then deduce the following lemma. It is a ``robust version'' of the fact that if a list of vectors does not contain $r$ linearly independent vectors, then among the vectors on this list we can find a linearly independent set of size less than $r$ whose span contains the entire list.

\begin{lem}\label{lem-choose-from-list}
Fix a positive integer $r$. Consider some $0< \eps\leq 1/2$, and consider a list of vectors $v_1,\dots,v_k\in \FF^n$ such that $\Vert v_i\Vert_\infty\le 1$ for all $i$. Suppose that there is no subset of $r$ vectors from this list which are $\eps^{(6r)^r}$-independent. Then for some $0\leq q\leq r-1$, we can choose vectors $w_1,\dots,w_{q}$ among the list $v_1,\dots,v_k$ such that both of the following conditions are satisfied:
\begin{compactitem}
\item[(i)] $w_1,\dots, w_q$ are $\eps^{(6r)^q}$-independent, and
\item[(ii)]for each $i=1,\dots,k$, there is a vector $v_i'\in \spn(w_1,\dots,w_q)$ with $\Vert v_i-v_i'\Vert_1\leq O(\eps^{5r\cdot (6r)^q})\cdot n$.
\end{compactitem}
\end{lem}

Here the implicit constant in the $O$-term may depend on $r$.

\begin{proof}
Let us choose a collection of $\eps^{(6r)^q}$-independent vectors $w_1,\dots,w_{q}$ among the list $v_1,\dots,v_k$, with $q\in \{0,\dots,r\}$ as large as possible (this is well-defined, since $q=0$ is definitely possible). By our assumption on $v_1,\dots,v_k$ we must have $q\leq r-1$.

We need to check condition (ii). Note that for all $i$ for which $v_i$ is one of the vectors $w_1,\dots,w_q$, condition (ii) holds trivially with $v_i'=v_i$. For all other $1\leq i\leq k$, the $q+1$ vectors $w_1,\dots,w_{q}, v_i$ cannot be $\eps^{(6r)^{q+1}}$-independent by maximality of $q$, so by \cref{lem-indep-or-close-lin-comb} there is $v_i'\in \spn(w_1,\dots,w_q)$ with
\[\Vert v_i-v_i'\Vert_1\leq O\left(\left(\eps^{(6r)^q}\right)^{-q}\cdot \eps^{(6r)^{q+1}}\right)\cdot n=O\left(\eps^{5r\cdot (6r)^q}\right)\cdot n.\tag*{\qedhere}\]
\end{proof}

Finally, the following lemma is not strictly about $\eps$-independence but will be used several times in the proofs of \cref{lem-matrix-rank-close,lem-matrix-rank-close-set}. For given vectors $w_1,\dots,w_q$, the lemma is about finding a vector $v\in \spn(w_1,\dots,w_q)$ with certain prescribed coordinates, and controlling the distance of $v$ to another given vector $\tilde{v}$ in this span.

\begin{lem}\label{claim-linear-comb-prescribed-entries}
Fix a non-negative integer $q$. Let $w_1,\dots,w_q\in \FF^n$ be vectors satisfying $\Vert w_i\Vert_\infty\le 1$ for each $i$, and let $M$ be the $q\times n$ matrix whose rows are the vectors $w_1,\dots,w_q$. Consider some subset $I\su [n]$ of size $q$, and suppose that the $q\times q$ matrix $M_I$ satisfies $\det M_I\ne 0$.

Now, for each $i\in I$ let us specify a value $v^{(i)}\in \FF$. Then there is a unique vector $v\in \spn(w_1,\dots,w_q)$ having the prescribed values $v^{(i)}$ in the coordinates $i\in I$. Furthermore, for any vector $\tilde{v}\in \spn(w_1,\dots,w_q)$ we have
\[\Vert v-\tilde{v}\Vert_1\leq  O(|\det M_I|^{-1})\cdot n\cdot \Vert\tilde v_I-v_I\Vert_1.\]
\end{lem}

Here the implicit constant in the $O$-term may depend on $q$.

\begin{proof}
The existence and uniqueness of $v$ follow directly from the fact that $\det M_I\ne 0$. Then, consider some vector $\tilde{v}\in \spn(w_1,\dots,w_q)$, and write $v-\tilde{v}=a_1w_1+\dots+a_qw_q$ for some $a_1,\dots,a_q\in \FF$.

Note that we can determine the coefficients $a_1,\dots, a_q$ by the equation $(a_1,\dots,a_q)M_I=(v-\tilde{v})_I$, where we interpret $(v-\tilde{v})_I\in \FF^I$ as a row vector. In other words, we have $(a_1,\dots,a_q)=(v-\tilde{v})_IM_I^{-1}$. The entries of $M_I^{-1}$ have absolute value $O(|\det M_I|^{-1})$ (by the formula for the inverse of a matrix in terms of its determinant and its adjugate). Thus, the absolute values of $a_1,\dots,a_q$ are of the form $O(|\det M_I|^{-1})\cdot \Vert\tilde v_I-v_I\Vert_1$. Since $\Vert w_i\Vert_\infty\le 1$ for each $i$, we conclude 
\[\Vert v-\tilde{v}\Vert_1=\Vert a_1w_1+\dots+a_qw_q\Vert_1\leq (\vert a_1\vert+\dots+\vert a_q\vert)\cdot n\leq  O(|\det M_I|^{-1})\cdot n\cdot \Vert\tilde v_I-v_I\Vert_1.\tag*{\qedhere}\]
\end{proof}

\section{Proving closeness to a low-rank symmetric matrix}\label{sec:lem-matrix-rank-close}

\subsection{Proof of \texorpdfstring{\cref{lem-matrix-rank-close}}{Lemma~\ref{lem-matrix-rank-close}}}

In this section, we finally prove \cref{lem-matrix-rank-close}: given a symmetric matrix $A$ which does not have $\delta n$ disjoint $r$-tuples of $\alpha$-independent rows, we show that there is a symmetric matrix $H$ that is close to $A$ (in terms of the entrywise $L_1$-norm) and has rank less than $r$. We outline the approach with a sequence of claims, whose proofs we will provide afterwards. We assume that $\alpha$ is sufficiently small, and for all asymptotic notation we treat $r$ as a constant.

First, we find a symmetric matrix $A^*\in \FF^{n\times n}$ which is close to $A$ and does not have any $r$-tuple of $\alpha$-independent rows, as in the following claim.

\begin{claim}\label{claim-matrix-A-star}
Consider the setting of \cref{lem-matrix-rank-close}. Then we can find a symmetric matrix $A^*\in \FF^{n\times n}$ with $\Vert A^*\Vert_\infty\le 1$, such that $\Vert A^*-A\Vert_1\leq O(\delta)\cdot n^2$ and such that the matrix $A^*$ does not have any $r$-tuple of $\alpha$-independent rows.
\end{claim}

Second, we use \cref{lem-choose-from-list} to identify a subset of rows $w_1,\dots,w_q$ of $A^*$, where $0\leq q\leq r-1$, such that every row of $A^*$ can be approximated by a linear combination of $w_1,\dots,w_q$. We then form a matrix $B^*$ by replacing the rows of $A^*$ by these approximations. This will give the following.

\begin{claim}\label{claim:close-to-special-low-rank}
Let the matrix $A^*$ be as in \cref{claim-matrix-A-star}. For some $0\leq q\leq r-1$ and $0<\tilde{\alpha}\le \alpha^{(6r)^{-r}}$ we can find $\tilde\alpha$-independent rows $w_1,\dots,w_q$ of $A^*$, and a matrix $B^*\in \FF^{n\times n}$, such that each row of $B^*$ lies in $\spn(w_1,\dots,w_q)$ and such that $\Vert \row_i(A^*)-\row_i(B^*)\Vert_1\le O(\tilde \alpha^{4r})\cdot n$ for all $i\in [n]$.
\end{claim}

\cref{claim:close-to-special-low-rank} ensures that each row of $A^*$ is close to the corresponding row of $B^*$. However, we have no control over the columns of $A^*$ and $B^*$. In the next step, we ``zero out'' some rows and columns of $A^*$ and $B^*$ to obtain matrices $A'$ and $B'$ such that every row and column of $A'$ is close to the corresponding row or column of $B'$.

\begin{claim}\label{claim:zero-some-columns}
Consider $A^*$, $B^*$, $\tilde\alpha$ and $w_1,\dots,w_q$ as in \cref{claim:close-to-special-low-rank}. We can find a symmetric matrix $A'\in \FF^{n\times n}$, a matrix $B'\in \FF^{n\times n}$, and $(\tilde \alpha/2)$-independent vectors $w_1',\dots,w_q'\in \FF^n$ such that each row of $B'$ lies in $\spn(w_1',\dots,w_q')$, such that each $\Vert w_i'\Vert_\infty\le 1$, and such that we have
\[\Vert A'-A^*\Vert_1\le O(\tilde \alpha^{r})\cdot n^2\]
and, for each $i\in [n]$, 
\[\Vert \row_i(A')-\row_i(B')\Vert_1\le O(\tilde{\alpha}^{3r})\cdot n \quad\text{ and }\quad\Vert \col_i(A')-\col_i(B')\Vert_1 \le O(\tilde{\alpha}^{3r})\cdot n.\]
\end{claim}

While \cref{claim:zero-some-columns} ensures that each row or column of $A'$ is close to the corresponding row or column of $B'$, it does not give control over individual entries of $A'$. However, the following claim (proved using \cref{lem-large-det-minor-graph-version}) states that we can find a subset $I\su [n]$ such that for all distinct $i,j\in I$, the $(i,j)$-entry of $A'$ is close to the $(i,j)$-entry of $B'$.

\begin{claim}\label{claim:graph-lemma-application}
Let $A'=(a'_{ij})_{i,j}$, $B'=(b'_{ij})_{i,j}$ and $w_1',\dots,w_q'$ be as in \cref{claim:zero-some-columns}, and let $M$ be the $q\times n$ matrix whose rows are $w_1',\dots,w_q'$. There is a subset $I\subseteq [n]$ of size $q$ such that for all distinct $i,j\in I$ we have $|a'_{ij}-b'_{ij}|<\tilde \alpha^{2r}$, and such that the $q\times q$ matrix $M_I$ satisfies $|\det M_I|\ge \Omega(\tilde \alpha^{r-1})$.
\end{claim}

We can then use this subset $I\su [n]$ to construct our final matrix $H$, in the following claim.

\begin{claim}\label{claim:final-symmetric}
Let $A'=(a'_{ij})_{i,j}$, $B'=(b'_{ij})_{i,j}$ and $w_1',\dots,w_q'$ be as \cref{claim:zero-some-columns} and let $I\su [n]$ be as in \cref{claim:graph-lemma-application}. Define $h_{ii}=b'_{ii}$ for all $i\in I$ and $h_{ij}=a'_{ij}$ for all distinct $i,j\in I$. Then we can extend these values to a symmetric matrix $H=(h_{ij})\in \FF^{n\times n}$ such that every row of $H$ is in $\spn(w_1',\dots,w_q')$ and such that $\Vert H-B'\Vert\le  O(\tilde\alpha)\cdot n^2$.
\end{claim}

The proof of \cref{claim:final-symmetric} will require \cref{claim-linear-comb-prescribed-entries} and the following technical lemma.

\begin{lem}\label{lem:funny-matrix-symmetric}
Consider vectors $v_1,\dots, v_q\in \FF^n$, and let $M$ be the matrix with these vectors as rows. Consider a subset $I\subseteq [n]$ with size $q$, such that the $q\times q$ matrix $M_I$ is invertible. Let $H=(h_{ij})_{i,j}\in \FF^{n\times n}$ be a matrix each of whose rows is in $\spn(v_1,\dots,v_q)$, such that $h_{ij}=h_{ji}$ for all $i\in I$ and all $j\in [n]$. Then $H$ is symmetric.
\end{lem}

It is not hard to see that \cref{claim-matrix-A-star,claim:close-to-special-low-rank,claim:close-to-special-low-rank,claim:zero-some-columns,claim:graph-lemma-application,claim:final-symmetric} imply \cref{lem-matrix-rank-close}. Indeed, the symmetric matrix $H$ in \cref{claim:final-symmetric} clearly has rank at most $q\leq r-1$. Furthermore, as $\Vert A'-B'\Vert_1=\sum_{i=1}^n \Vert \row_i(A')-\row_i(B')\Vert_1\leq  O(\tilde{\alpha}^{3r})\cdot n^2$ by \cref{claim:zero-some-columns}, we obtain
\[\Vert H-A\Vert_1\leq \Vert H-B'\Vert_1+\Vert B'-A'\Vert_1+\Vert A'-A^*\Vert_1+\Vert A^*-A\Vert_1\leq O(\tilde{\alpha}+\delta)\cdot n^2=O(\alpha^{(6r)^{-r}}+\delta)\cdot n^2.\]

It remains to prove \cref{claim-matrix-A-star,claim:close-to-special-low-rank,claim:zero-some-columns,claim:graph-lemma-application,claim:final-symmetric} and \cref{lem:funny-matrix-symmetric}.

\begin{proof}
[Proof of \cref{claim-matrix-A-star}]
By assumption, there do not exist $\delta n$ disjoint $r$-tuples of $\alpha$-independent rows of $A$. Choose a maximal collection of such $r$-tuples and let $J\subseteq [n]$ be the set of all rows involved (so $|J|\le r\delta n$). Let $A^*=(a^*_{ij})_{i,j}\in \FF^{n\times n}$ be the symmetric matrix obtained from $A$ by setting to zero all rows and all columns with indices in $J$. As $\Vert A\Vert_\infty\le 1$, we have $
\Vert A^*-A\Vert_1\leq 2 \vert J\vert n\leq 2r \delta  n^2=O(\delta)\cdot n^2$. We claim that $A^*$ does not have any $r$-tuple of $\alpha$-independent rows. Clearly, no $r$-tuple containing a zero row can be $\alpha$-independent. For any $r$-tuple of rows of $A^*$ with indices in $[n]\setminus J$, the corresponding $r$-tuple of rows in $A$ must be $\alpha$-dependent (by the maximality of the collection chosen in the beginning). It is not hard to see that this $r$-tuple stays $\alpha$-dependent when setting to zero the columns with indices in $J$.
\end{proof}

\begin{proof}
[Proof of \cref{claim:close-to-special-low-rank}]
We can apply \cref{lem-choose-from-list} with $\eps=\alpha^{(6r)^{-r}}$ and the list of rows $v_1,\dots,v_n$ of $A^*$, to obtain $\alpha^{(6r)^{q-r}}$-independent rows $w_1,\dots,w_q$ of $A^*$ for some $0\leq q\leq r-1$, as well as a list of vectors $v_1',\dots,v_n'\in \spn(w_1,\dots,w_q)$ satisfying $\Vert v_i-v_i'\Vert_1\le O(\alpha^{5r(6r)^{q-r}})n\le O(\alpha^{4r(6r)^{q-r}})n$ for all $i\in [n]$ (the wasteful second inequality here will make it easier to explain how to adapt this proof to prove \cref{lem-matrix-rank-close-set} in \cref{sec:lem-matrix-rank-close-set}). Then define $\tilde\alpha=\alpha^{(6r)^{q-r}}\leq \alpha^{(6r)^{-r}}$ and let $B^*\in \FF^{n\times n}$ be the matrix with rows $v_1',\dots,v_n'$.
\end{proof}

\begin{proof}
[Proof of \cref{claim:zero-some-columns}]
Note that by the assumptions on $A^*$ and $B^*$, we have
\[\sum_{i=1}^n \Vert \col_i(A^*)-\col_i(B^*)\Vert_1=\Vert A^*-B^*\Vert_1=\sum_{i=1}^n \Vert \row_i(A^*)-\row_i(B^*)\Vert_1\leq  O(\tilde{\alpha}^{4r})\cdot n^2.\]
Defining $J\su [n]$ be the set of those indices $j\in [n]$ such that $\Vert \col_j(A^*)-\col_j(B)\Vert_1\geq \tilde{\alpha}^{3r}\cdot n$, we obtain that $\vert J\vert\leq O(\tilde \alpha^r)\cdot n$. Let the matrices $A'$ and $B'$ be obtained from $A^*$ and $B^*$ by setting to zero all the entries in all rows and columns with indices in $J$. Similarly, let the vectors $w_1',\dots w_q'$ be obtained from $w_1,\dots,w_q$ by setting the entries with indices in $J$ to zero. Since each row of $B^*$ lies in $\spn(w_1,\dots,w_q)$, each row of $B'$ lies in $\spn(w_1',\dots,w_q')$ (note that this is trivially true for the all-zero rows with indices in $J$).

Note that $A'$ is symmetric, since $A^*$ is symmetric. Furthermore, $\Vert A^*\Vert_\infty\le 1$, so $\Vert w_i\Vert_\infty\le 1$ and $\Vert w_i'\Vert_\infty\le 1$ for each $i$. We claim that $w_1',\dots,w_q'$ are $(\tilde\alpha/2)$-independent. For $q=0$ this is trivially true, so we may assume that $q\geq 1$ and therefore $r\geq 2$. Recall that $w_1,\dots,w_q$ are $\tilde\alpha$-independent, and that we changed only $r|J|=O(\tilde{\alpha}^{r})\cdot n\le (\tilde\alpha/2)\cdot n$ entries of $w_1,\dots,w_q$ to zero to obtain $w_1',\dots,w_q'$. Thus, the vectors $w_1',\dots,w_q'$ are indeed $(\tilde\alpha/2)$-independent.

We have $\Vert A'-A^*\Vert_1\le 2\vert J\vert n\leq  O(\tilde \alpha^{r})\cdot n^2$. For each $i\in J$, we have $\row_i(A^*)=0=\row_i(B^*)$ and $\col_i(A^*)=0=\col_i(B^*)$. On the other hand, for $i\in [n]\setminus J$, by the properties in \cref{claim:close-to-special-low-rank} we have $\Vert \row_i(A')-\row_i(B')\Vert_1\le \Vert \row_i(A^*)-\row_i(B^*)\Vert_1\le O(\tilde{\alpha}^{4r})\cdot n$  and by the definition of $J$ we have $\Vert \col_i(A')-\col_i(B')\Vert_1\le \Vert \col_i(A^*)-\col_i(B^*)\Vert_1< \tilde{\alpha}^{3r}\cdot n$.
\end{proof}

\begin{proof}
[Proof of \cref{claim:graph-lemma-application}]
Note that the case $q=0$ is trivial, so we may assume that $q\geq 1$ and therefore $r\geq 2$. Consider the graph $G$ on the vertex set $[n]$ where for any $1\leq i<j\leq n$ we draw an edge between the vertices $i$ and $j$ if $\vert a'_{ij}-b'_{ij}\vert\geq \tilde{\alpha}^{2r}$ or if $\vert a'_{ji}-b'_{ji}\vert\geq \tilde{\alpha}^{2r}$. By the last part of \cref{claim:zero-some-columns}, this graph has maximum degree $O(\tilde{\alpha}^{r})\cdot n$, so the desired result follows from \cref{lem-large-det-minor-graph-version} with $\eps=\tilde{\alpha}/4$ and $\delta=O(\tilde\alpha^r)$ (using that $\eps+q(q-1)\delta\leq \tilde{\alpha}/4+O(\tilde\alpha^r)\leq \tilde\alpha/2$ as $r\geq 2$, and also recalling that $q\leq r-1$).
\end{proof}

\begin{proof}
[Proof of \cref{claim:final-symmetric}]
Recall that for $i\in I$, we defined $h_{ii}=b'_{ii}$, and for distinct $i,j\in I$ we defined $h_{ij}=a'_{ij}$. For every $i\in I$, we can uniquely extend the vector $(h_{ij})_{j\in I}\in \FF^I$ to a vector $\row_i(H)=(h_{ij})_{1\le j\le n}\in \spn(w_1',\dots,w_q')$, using \cref{claim-linear-comb-prescribed-entries}. Then, using that every row of $B'$ is also in $\spn(w_1',\dots,w_q')$, by the second part of \cref{claim-linear-comb-prescribed-entries} we have (recalling the defining properties of $I$ in \cref{claim:graph-lemma-application})
\begin{equation}
\Vert \row_i(H)-\row_i(B')\Vert_1 \leq  O(|\det M_I|^{-1})\cdot n\cdot \sum_{j\in I\setminus\lbrace i\rbrace}\vert a'_{ij}-b'_{ij}\vert\le O(\tilde{\alpha}^{1-r})\cdot n\cdot q\cdot \tilde\alpha^{2r}=O(\tilde\alpha^{r+1})\cdot n.\label{eq-row-H-row-B}
\end{equation}
for every $i\in I$. So far we have defined $h_{ij}$ for $i\in I$ and $j\in [n]$. Since $A'$ is symmetric, we have $h_{ij}=a_{ij}'=a_{ji}'=h_{ji}$ for distinct $i,j\in I$. Now, for $j\in I$ and $i\in [n]\setminus I$, let us define $h_{ij}=h_{ji}$. Then, for $i\in [n]\setminus I$, we proceed very similarly to before: we can uniquely extend the vector $(h_{ij})_{j\in I}\in \FF^I$ to a vector $\row_i(H)=(h_{ij})_{1\le j\le n}\in \spn(w_1',\dots,w_q')$. The resulting matrix $H=(h_{ij})_{i,j}$ satisfies $h_{ij}=h_{ji}$ for all $i\in I$ and $j\in [n]$ and all of its rows lie in $\spn(w_1',\dots,w_q')$. So by \cref{lem:funny-matrix-symmetric}, $H$ is therefore symmetric. Furthermore, for all $i\in [n]$ by the second part of \cref{claim-linear-comb-prescribed-entries} we have
\[\Vert \row_i(H)-\row_i(B')\Vert_1\leq O(\tilde{\alpha}^{1-r})\cdot n\cdot \sum_{j\in I}\vert h_{ij}-b'_{ij}\vert\]
and therefore we obtain
\[\Vert H-B'\Vert_1=\sum_{i=1}^n \Vert \row_i(H)-\row_i(B')\Vert_1\leq O(\tilde{\alpha}^{1-r})\cdot n\cdot \sum_{i=1}^n\sum_{j\in I}\vert h_{ij}-b'_{ij}\vert=O(\tilde{\alpha}^{1-r})\cdot n\cdot \Vert (H-B')_I\Vert_1.\]

For an $n\times n$ matrix $B$, let $B^I$ denote the submatrix consisting of the rows indexed by $I$. By the properties of $A'$ and $B'$ in \cref{claim:zero-some-columns} we have
\[\Vert (A'-B')^I\Vert_1\le O(\tilde \alpha^{3r})\cdot n\quad\text{ and }\quad \Vert (A'-B')_I\Vert_1\le O(\tilde \alpha^{3r})\cdot n,\]
and by symmetry of $H$ and $A'$, we have $\Vert (H-A')_I\Vert_1=\Vert (H-A')^I\Vert_1$, so
\begin{align*}
\Vert (H-B')_I\Vert_1&\leq \Vert (H-A')_I\Vert_1+\Vert (A'-B)_I\Vert_1\\
&\le \Vert (H-A')^I\Vert_1+O(\tilde \alpha^{3r})\cdot n\le\Vert (H-B')^I\Vert_1+O(\tilde \alpha^{3r})\cdot n.
\end{align*}
On the other hand, from \cref{eq-row-H-row-B} we obtain $\Vert (H-B')^I\Vert_1\leq O(\tilde\alpha^{r+1})\cdot n$, so it follows that $\Vert (H-B')_I\Vert_1\leq O(\tilde\alpha^{r+1})\cdot n$ and therefore $\Vert H-B'\Vert_1\le O(\tilde\alpha^{1-r})\cdot n\cdot O(\tilde\alpha^{r+1})\cdot n=O(\tilde\alpha^2)\cdot n^2$.
\end{proof}

\begin{proof}
[Proof of \cref{lem:funny-matrix-symmetric}]
Write $v_{ij}$ for the $j$th component of $v_i$. Without loss of generality we may assume that $I=[q]\su [n]$. Also, we may assume that $M_I$ is the $q\times q$ identity matrix (we can replace $v_1,\dots,v_q$ by different vectors with the same span).

Now, each row of $H$ is a linear combination of $v_1,\dots,v_q$, and given the above assumptions it is easy to read off the coefficients: $\row_i(H)=h_{i1}v_1+\dots +h_{iq}v_q$ for all $i\in [n]$. So, using the assumption that $h_{ik}=h_{ki}$ for all $k\in I=[q]$ and all $i\in [n]$, we have
\[h_{ij}=\sum_{k=1}^q h_{ik}v_{kj}=\sum_{k=1}^q h_{ki}v_{kj}=\sum_{k=1}^q\left(\sum_{\ell=1}^q h_{k\ell}v_{\ell i}\right) v_{kj}=\sum_{ k,\ell\in [q]}h_{k\ell}v_{\ell i} v_{kj}\]
for all $i,j\in [n]$. (For the third equality, we read off the coefficients for $\row_k(H)$ in the same way we read off the coefficients for $\row_i(H)$.) This expression is symmetric in $i$ and $j$, since $h_{k\ell}=h_{\ell k}$ for all $k,\ell\in [q]$.
\end{proof}

\subsection{Adapting the proof for \texorpdfstring{\cref{lem-matrix-rank-close-set}}{Lemma~\ref{lem-matrix-rank-close-set}}\label{sec:lem-matrix-rank-close-set}}

In this section we describe how to prove \cref{lem-matrix-rank-close-set}, by slightly modifying the proof of \cref{lem-matrix-rank-close} in the previous subsection. Specifically, given a finite set $S\subseteq \FF$ we define a finite set $S'=S'(r,S)$, and given a matrix $A$ with entries in $S$, we describe how to adapt the proof of \cref{lem-matrix-rank-close} to ensure that $H$ has entries in $S'$. 

\begin{defn}
\label{def:S'}
Given a non-negative integer $q$ and a finite set $S\su \FF$ with $0\in S$, define $T_q(S)\su \FF^r$ to be the set of those vectors $v\in \FF^r$ which are the solution to an equation of the form $Mv=z$ for some invertible matrix $M\in S^{q\times q}$ and some vector $z\in S^q$. Also, define
$\tau_q(S)=\lbrace v^T w: v\in T_q(S), w\in S^q\rbrace$. Finally, for a positive integer $r$, define $\sigma_r(S)=S\cup \tau_1(S)\cup\dots\cup \tau_{r-1}(S)$ and $S'(r,S)=\sigma_r(\sigma_r(\sigma_r(S)))$.
\end{defn}

Note that the sets $T_q(S)$, $\tau_q(S)$, $\sigma_r(S)$ and $S'(r,S)$ in \cref{def:S'} are finite.

Recall that in the proof of \cref{lem-matrix-rank-close} (more specifically, in the proof of \cref{claim:final-symmetric}) all rows of the matrix $H$ were chosen by applying \cref{claim-linear-comb-prescribed-entries} to find a vector in $\spn(w_1',\dots,w_q')$ with certain prescribed entries. In order to ensure that the entries of $H$ are in $S'$, we need a way to control the entries of the vectors found when applying \cref{claim-linear-comb-prescribed-entries}. The following lemma gives such control.

\begin{lem}\label{lem-lin-comb-in-tau}
Fix a non-negative integer $q$ and a finite set $S\su \FF$ with $0\in S$. Consider vectors $w_1,\dots,w_q\in S^n$ and a subset $I\su [n]$ of size $q$. Let $M$ be the $q\times n$ matrix whose rows are the vectors $w_1,\dots,w_q$, and suppose that its $q\times q$ submatrix $M_I$ is invertible. Finally, consider $v=(v^{(1)},\dots, v^{(n)})\in \spn(w_1,\dots,w_q)$ such that $v^{(i)}\in S$ for all $i\in I$. Then $v\in (\tau_q(S))^n$. In particular, if $r$ is a positive integer such that $q\leq r-1$, we have $v\in (\sigma_r(S))^n$.
\end{lem}

Note that \cref{lem-lin-comb-in-tau} implies the following: if we apply \cref{claim-linear-comb-prescribed-entries} to vectors $w_1,\dots,w_q$ with entries in $S$, and the prescribed values $v^{(i)}$ also all lie in $S$, then the resulting vector $v$ in \cref{claim-linear-comb-prescribed-entries} satisfies $v\in (\tau_q(S))^n\su (\sigma_r(S))^n$.

\begin{proof}[Proof of \cref{lem-lin-comb-in-tau}]
Write $v=\lambda_1w_1+\dots+\lambda_qw_q$ for some $\lambda=(\lambda_1,\dots,\lambda_q)\in \FF^q$. Only considering the coordinates with indices in $I$, this implies $v_I=M_I^T\lambda$. We have  $v_I\in S^q$, and $M_I^T\in S^{q\times q}$ is an invertible matrix, so $\lambda\in T_q(S)$. Now, for each $i\in [n]$, the entry $v^{(i)}$ is the product of the row vector $\lambda=(\lambda_1,\dots,\lambda_n)$ with the column vector formed by the $i$-th coordinates of $w_1,\dots,w_q$ (this column vector is in $S^q$). Hence $v^{(i)}\in \tau_q(S)\subseteq \sigma_r(S)$ for all $i\in [n]$.
\end{proof}

In order to control the entries of the rows of $H$ when applying \cref{claim-linear-comb-prescribed-entries} together with \cref{lem-lin-comb-in-tau}, we clearly also need to control the entries we prescribe. These prescribed entries ultimately depend on the entries of $A'$ and $B'$ (see our definition of $h_{ij}$ for $i,j\in I$ in \cref{claim:final-symmetric}). Each entry of $A'$ is also an entry of $A$ or equals zero (by the way we constructed $A'$ and $A^*$ in the proofs of \cref{claim-matrix-A-star,claim:close-to-special-low-rank}), so we can easily control the entries of $A'$ (and similarly $A^*$). However, in order to control the entries of $B'$, we need to control the entries of $B^*$, which were obtained by applying \cref{lem-choose-from-list} to the list of rows of $A^*$.

The following lemma will be used in combination with \cref{lem-choose-from-list}. While it does not give direct control over the vectors $v_i'\in \spn(w_1,\dots,w_q)$ obtained from \cref{lem-choose-from-list} as approximations of the rows of $A^*$, it states that we can find slightly weaker approximations ensuring that the entries of these approximating vectors lie in a certain set.

\begin{lem}\label{lem-lin-comb-in-tau-close}
Fix a non-negative integer $q$ and a finite set $S\su \FF$ with $0\in S$ and $\vert s\vert\leq 1$ for all $s\in S$. Consider $0<\alpha\leq 1$ and $0<\eta\leq 1$, as well as $\alpha$-independent vectors $w_1,\dots,w_q\in S^n$ and a vector $v\in S^n$, such that there is some vector $\tilde{v}\in \spn(w_1,\dots,w_q)$ with $\Vert v-\tilde{v}\Vert_1\leq \eta\cdot n$. Then there is a vector $v^*\in (\tau_{q}(S))^n\cap \spn(w_1,\dots,w_q)$ with $\Vert v-v^*\Vert_1\leq O(\alpha^{-(q+1)})\cdot \eta \cdot n$.
\end{lem}

\begin{proof}
Note that for $q=0$ the statement is trivial: Indeed, we must have $\tilde{v}=0$ and can therefore take $v^*=\tilde{v}=0\in (\tau_{0}(S))^n\cap \spn(w_1,\dots,w_q)$. So let us from now on assume that $q\geq 1$.

Let $v=(v^{(1)},\dots, v^{(n)})\in S^n$ and $\tilde{v}=(\tilde{v}^{(1)},\dots, \tilde{v}^{(n)})\in \FF^n$. Since $\Vert v-\tilde{v}\Vert_1\leq \eta\cdot n$, there are at most $(\alpha/(2q))\cdot n$ indices $i\in [n]$ with $\vert \tilde{v}^{(i)}-v^{(i)}\vert\geq 2q\alpha^{-1}\eta$. Let $J$ be the set of those indices $i\in [n]$, then $\vert J\vert\leq (\alpha/(2q))\cdot n$. Then, let $w_1',\dots,w_q'$ be obtained from $w_1,\dots,w_q$ by setting to zero all entries with coordinates in $J$. Since all entries of the $\alpha$-independent vectors $w_1,\dots,w_q$ have absolute value at most $1$, and $\vert J\vert\cdot q\le (\alpha/2) \cdot n$, these new vectors $w_1',\dots,w_q'$ are $(\alpha/2)$-independent.

Let $M$ (respectively $M'$) be the $q\times n$ matrix with $w_1,\dots,w_q$ (respectively $w_1',\dots,w_q'$) as rows. By \cref{lem-large-det-minor}, there exists a subset $I\su [n]$ of size $q$ such that $|\det M'_I|\ge (\alpha/2)^q$. Note that $I\cap J=\emptyset$, since if $M'_I$ had a zero column it would have zero determinant. Hence $M_I=M'_I$ and $|\det M_I|\ge (\alpha/2)^q$.

Now we can apply \cref{claim-linear-comb-prescribed-entries} to obtain a vector $v^*\in \spn(w_1,\dots,w_q)$ such that for each $i\in I$ the $i$-th entry of $v^*$ equals $v^{(i)}\in S$. By \cref{lem-lin-comb-in-tau} we have $v^*\in (\tau_{q}(S))^n$, and by the second part of \cref{claim-linear-comb-prescribed-entries} we have (using that $\vert \tilde{v}^{(i)}-v^{(i)}\vert< 2q \alpha^{-1} \eta$ for all $i\in I\su [n]\setminus J$)
\[\Vert v^*-\tilde{v}\Vert_1\leq  O(|\det M_I|^{-1})\cdot n\cdot \sum_{i\in I} \vert \tilde{v}^{(i)}-v^{(i)}\vert\leq O(\alpha^{-q})\cdot n\cdot q\cdot 2q\alpha^{-1}\eta =O(\alpha^{-(q+1)})\cdot \eta \cdot n.\]
Hence $\Vert v-v^*\Vert_1\leq \Vert v-\tilde{v}\Vert_1+\Vert \tilde{v}-v^*\Vert_1\leq O(\alpha^{-(q+1)})\cdot \eta \cdot n$.
\end{proof}

Now, using \cref{lem-lin-comb-in-tau} and \cref{lem-lin-comb-in-tau-close}, it is not hard to modify the proof of \cref{lem-matrix-rank-close} to prove \cref{lem-matrix-rank-close-set}.

\begin{proof}
[Proof of \cref{lem-matrix-rank-close-set}]
Let $r$, $S$ and $A$ be as in \cref{lem-matrix-rank-close-set}. Then $A$ satisfies the assumptions in \cref{lem-matrix-rank-close} with the additional property that all entries of $A$ are in the set $S$. Let $S'=S'(r,S)$ be as in \cref{def:S'} (note that we may assume that $0\in S$, since otherwise we can replace $S$ by $S\cup \lbrace 0\rbrace$).

We can now proceed as in the proof of \cref{lem-matrix-rank-close} in the previous subsection. First, without making any changes to the proof of \cref{claim-matrix-A-star}, we observe that each entry of $A^*$ is either zero (which we are assuming is in $S$), or is the same as the corresponding entry in $A$ (which is also in $S$). Second, we consider the matrix $B^*$ in \cref{claim:close-to-special-low-rank}. In the proof of \cref{lem-matrix-rank-close} we applied \cref{lem-choose-from-list} to the rows $v_1,\dots,v_n\in S^n$ of $A^*$ to obtain a list of vectors $v_1',\dots,v_n'\in \spn(w_1,\dots,w_q)$ satisfying $\Vert v_i-v_i'\Vert_1\le O(\alpha^{5r(6r)^{q-r}})n=O(\tilde{\alpha}^{5r})n$ for all $i\in [n]$. We then took these vectors as the rows of $B^*$. For the proof of \cref{lem-matrix-rank-close-set} we need an additional step: for each $i$ we apply \cref{lem-lin-comb-in-tau-close} with $v=v_i$ and $\tilde v=v_i'$, which gives a vector $v_i^*\in (\tau_{q}(S))^n\cap\spn(w_1,\dots,w_q)$ with $\Vert v_i-v_i^*\Vert_1\leq O(\tilde\alpha^{-(q+1)})\cdot O(\tilde{\alpha}^{5r}) \cdot n\leq O(\tilde{\alpha}^{4r})n$ (since $q\leq r-1$). We can then take $v_1^*,\dots,v_n^*$ as the rows of $B^*$, so that all entries of $B^*$ lie in the set $\tau_{q}(S)\su \sigma_r(S)$.

Next, without making any modifications to the proof of \cref{claim:zero-some-columns}, observe that the matrices $A'$ and $B'$ have entries are in $\sigma_{r}(S)$ (the entries of $A'$ and $B'$ were chosen to be zeroes or entries of $A^*$ or $B^*$). Finally, when constructing the matrix $H=(h_{ij})_{i,j}$ as in the proof of \cref{claim:final-symmetric}, observe that the entries of $H$ all lie in $\sigma_r(\sigma_r(\sigma_r(S)))=S'$. Indeed, the entries $h_{ij}$ for $i,j\in I$ are in $\sigma_{r}(S)$. We obtained $\row_i(H)$ for $i\in I$ by applying \cref{claim-linear-comb-prescribed-entries} to these entries, so by \cref{lem-lin-comb-in-tau} all entries $h_{ij}$ with $i\in I$ and $j\in [n]$ lie in $\sigma_r(\sigma_r(S))$. Then we defined $h_{ij}=h_{ji}$ for $i\in [n]\setminus I$ and $j\in I$, and applied \cref{claim-linear-comb-prescribed-entries} again to obtain $\row_i(H)$ for $i\in [n]\setminus I$. Again by \cref{lem-lin-comb-in-tau} all entries of these rows lie in $\sigma_r(\sigma_r(\sigma_r(S)))$. Thus, all entries of $H$ lie in $\sigma_r(\sigma_r(\sigma_r(S)))$, as desired.
\end{proof}

\section{Concluding remarks}\label{sec:concluding}

There are still a number of future directions of research. Most
obviously, Costello's conjecture remains open, and there are several weakenings that
we think would already be very interesting to prove. For example,
can we remove the restriction on the coefficients in \cref{thm:rank-edit}? Can we prove any bound of
the form $n^{-1/2-\Omega\left(1\right)}$ on the point probabilities
of a real quadratic polynomial that is not close to splitting into
linear factors over the real numbers, even with strong assumptions
on the allowed coefficients?

Costello's conjecture also concerned higher-degree polynomials, and it would be interesting to investigate this direction as well. Iterating \cref{lem:complex-decoupling} does give a way to control higher-degree polynomials, but it seems unlikely that one could prove results as strong as \cref{thm:rank-L1,thm:rank-edit} without new ideas.

Regarding Ramsey graphs, it would be interesting if one could
remove the $o\left(1\right)$-term in \cref{thm:ramsey-KST}. It would also be very interesting to prove a bound on $\Pr(X=x)$ in terms of the variance of $X$, as follows. Such a bound would be best-possible due to Chebyshev's inequality.

\begin{conjecture}
\label{conj:KST}The following holds for any fixed constants $C,c>0$. Let $G$ be an $n$-vertex $C$-Ramsey graph, and, for some $cn\le k\le\left(1-c\right)n$, let $X$ be the
number of edges induced by a uniformly random subset of $k$ vertices
of $G$. Then for any $x\in\ZZ$,
we have
\[
\Pr\left(X=x\right)=O\left(\frac{1}{\sqrt{\Var X}}\right).
\]
\end{conjecture}

We remark that $\sqrt{\Var X}$ can be as small as $\Theta(n)$ (which is typical for a random graph $\GG\left(n,1/2\right)$), and can be as large as $\Theta(n^{3/2})$ (for example, consider a disjoint union of two random graphs $\GG\left(n/2,1/2\right)$ on $n/2$ vertices). We find \cref{conj:KST} particularly compelling because, if true, it would give a very simple, unified proof of the conjectures of Narayanan--Sahasrabudhe--Tomon and Erd\H os--Faudree--S\'os stated in the introduction, concerning subgraphs of Ramsey graphs with different numbers of edges. For example, consider an $n$-vertex $C$-Ramsey graph $G$ and consider some $k\in \NN$ satisfying $k=\Omega(n)$ and $n/2-k=\Omega(n)$. To prove the Erd\H os--Faudree--S\'os conjecture, it suffices to show that $G$ has $\Omega(n^{3/2})$ induced subgraphs with $k$ vertices and different numbers of edges. This can be shown using \cref{conj:KST}, as follows.
\begin{proof}[Sketch proof of the Erd\H os--Faudree--S\'os conjecture, assuming \cref{conj:KST}]
Let $G$ and $k$ be as above. Alon and Kostochka (see \cite[Equation~(2)]{AK09}) observed that there is a sequence of $2k$-vertex sets $U_0,\dots,U_t\su V(G)$ such that $e(U_t)-e(U_1)=\Omega(n^{3/2})$, but $|e(U_i)-e(U_{i-1})|\le n$ for all $1\le i\le t$. That is, we can ``loosely fill'' an interval of length $\Omega(n^{3/2})$ with edge-counts of $2k$-vertex induced subgraphs. This can be shown by using a discrepancy theorem to find two $2k$-vertex induced subgraphs whose numbers of edges differ by $\Omega(n^{3/2})$, and switching vertices one-by-one between these two subgraphs.

Now, let $X_i$ be the number of edges in a random subset of $k$ vertices of $G[U_i]$, and let $\supp(X_i)$ be the support of the random variable $X_i$. Note that by Chebyshev's inequality, more than half of the probability mass of each $X_i$ falls in the interval $I_i=[e(U_i)/2-2\sqrt{\Var(X_i)},e(U_i)/2+2\sqrt{\Var(X_i)}]$, which has length $O(\sqrt{\Var(X_i)})$. Since each $G[U_i]$ is itself an $O(1)$-Ramsey graph, \cref{conj:KST} would say that every individual point in this interval has probability mass only $O(1/\sqrt{\Var(X_i)})$, from which one can show that $|\supp(X_i)\cap I_i|=\Omega(\sqrt{\Var(X_i)})=\Omega(|I_i|)$. Given that $\sqrt{\Var(X_i)}=\Omega(n)$ for all $i$, the union $\bigcup_{i=0}^t I_i$ covers an $\Omega(1)$-fraction of the length-$\Omega(n^{3/2})$ interval between $e(U_0)/2$ and $e(U_t)/2$, which means that $|\bigcup_{i=0}^tI_i|= \Omega(n^{3/2})$. By a simple greedy algorithm, we can choose a disjoint collection $(I_i)_{i\in J}$ of intervals covering an $\Omega(1)$-fraction of $\bigcup_{i=0}^t I_i$. Then $\sum _{i\in J} |I_i|= |\bigcup_{i\in J}I_i|=\Omega(n^{3/2})$ and therefore $|\bigcup_{i\in J} \supp(X_i)|\geq \sum _{i\in J} |\supp(X_i)\cap I_i|=\sum _{i\in J} \Omega(|I_i|)=\Omega(n^{3/2})$. This gives $\Omega(n^{3/2})$ different edge-counts of $k$-vertex induced subgraphs.
\end{proof}
Actually, it is tempting to wonder whether one can strengthen \cref{conj:KST} even further, and prove a local central limit theorem estimating each of the point probabilities $\Pr(X=x)$ in terms of a Gaussian density function. With sufficiently
detailed understanding of the local behaviour of ``edge-statistic'' random variables in Ramsey graphs,
it might be possible to prove an old conjecture of Erd\H os and McKay~\cite{Erd92,Erd97}
that there is an interval $I=\left\{ 0,1,2\dots,\Omega\left(n^{2}\right)\right\} $
such that in every $O\left(1\right)$-Ramsey graph, for every $m\in I$
there is an induced subgraph with exactly $m$ edges (see \cite{AKS03}
for some progress on this question).

Finally, there may also be interesting directions of research concerning some of the auxiliary lemmas in this paper. For example, in \cref{cor:symmetric-close} we proved that if a matrix with bounded entries is close to being low-rank, and close to being symmetric, then it is close to being simultaneously low-rank and symmetric. Can we drop the assumption that the entries of the matrix are bounded? There may be a connection between results of this type and the existing literature on low-rank approximation of matrices.


\section*{Acknowledgments} 
We would like to thank Jacob Fox for insightful discussions, and Ragib Zaman for simplifying the proof of \cref{lem:L1-real-to-complex}.

\bibliographystyle{amsplain}


\begin{dajauthors}
\begin{authorinfo}[mkwa]
  Matthew Kwan\\
  Szeg\H o Assistant Professor\\
  Stanford University\\
  Stanford, USA\\
  mattkwan\imageat{}stanford\imagedot{}edu \\
  \url{http://web.stanford.edu/~mattkwan/}
\end{authorinfo}
\begin{authorinfo}[lsau]
  Lisa Sauermann\\
  Szeg\H o Assistant Professor\\
  Stanford University\\
  Stanford, USA\\
  lsauerma\imageat{}stanford\imagedot{}edu \\
  \url{https://web.stanford.edu/~lsauerma/}
\end{authorinfo}
\end{dajauthors}

\end{document}